\documentclass[12pt]{amsart}
\usepackage{amssymb,amsmath}
\usepackage{colordvi,xcolor}
\usepackage{graphicx}
\usepackage{enumerate}
\usepackage{mathrsfs}   
\usepackage{vmargin}
\setpapersize{USletter}
\setmargrb{2.5cm}{2.5cm}{2.5cm}{2.5cm} 

\numberwithin{equation}{section}
\newtheorem{theorem}{Theorem}[section]
\newtheorem{proposition}[theorem]{Proposition}
\newtheorem{lemma}[theorem]{Lemma}
\newtheorem{corollary}[theorem]{Corollary}
\theoremstyle{remark}
\newtheorem{definition}[theorem]{Definition}
\newtheorem{example}[theorem]{Example}
\newtheorem{remark}[theorem]{Remark}
\newcounter{FNC}[page]
\def\fauxfootnote#1{{\addtocounter{FNC}{2}\Magenta{$^\fnsymbol{FNC}$}%
     \let\thefootnote\relax\footnotetext{\Magenta{$^\fnsymbol{FNC}$#1}}}}
\newcommand{\MySmVec}[2]{(\mbox{\tiny$\begin{smallmatrix}#1\\#2\end{smallmatrix}$})}
\newcommand{\MySmTVec}[3]{\bigl(\mbox{\tiny$\begin{smallmatrix}#1\\#2\\#3\end{smallmatrix}$}\bigr)}
\newcommand{\MyVect}[2]{(\begin{smallmatrix}#1\\#2\end{smallmatrix})}
\newcommand{\MyDiamond}{\includegraphics{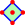}}
\newcommand{\MySmDiamond}{\includegraphics[height=6pt]{figures/Diamond.eps}}
\renewcommand{\qed}{\hfill\raisebox{-3pt}{\MyDiamond}}
\newcommand{\tsp}{\hspace{7pt}}

\newcommand{\defcolor}[1]{\RoyalBlue{#1}}
\newcommand{\demph}[1]{\defcolor{{\sl #1}}}

\newcommand{\lhra}{\ensuremath{\lhook\joinrel\relbar\joinrel\relbar\joinrel\rightarrow}}

\newcommand{\calA}{{\mathcal A}}
\newcommand{\calB}{{\mathcal B}}
\newcommand{\calC}{{\mathcal C}}
\newcommand{\calF}{{\mathcal F}}
\newcommand{\calP}{{\mathcal P}}
\newcommand{\calV}{{\mathcal V}}

\newcommand{\frakm}{{\mathfrak m}}

\newcommand{\bzero}{{\mathbf 0}}
\newcommand{\be}{{\mathbf e}}
\newcommand{\bff}{{\mathbf f}}

\DeclareMathOperator{\diag}{\rm diag}
\DeclareMathOperator{\supp}{\rm supp}
\DeclareMathOperator{\conv}{\rm conv}
\DeclareMathOperator{\spec}{\rm spec}

\DeclareMathOperator{\Aff}{\rm Aff}
\DeclareMathOperator{\rank}{\rm rank}
\DeclareMathOperator{\New}{\rm New}
\DeclareMathOperator{\Trop}{\rm Trop}
\DeclareMathOperator{\HP}{\rm HP}
\DeclareMathOperator{\HF}{\rm HF}

\newcommand{\C}{{\mathbb C}}
\newcommand{\N}{{\mathbb N}}
\renewcommand{\P}{{\mathbb P}}
\newcommand{\Q}{{\mathbb Q}}
\newcommand{\R}{{\mathbb R}}
\newcommand{\Z}{{\mathbb Z}}

\newcommand{\Vol}{\mbox{\rm Vol}}
\newcommand{\ini}{\mbox{\rm in}}
\newcommand{\Hom}{\mbox{\rm Hom}}
\newcommand{\MV}{\mbox{\it MV}}

\title{Ibadan Lectures on Toric Varieties}
\author{Frank Sottile}
\address{Frank Sottile\\
         Department of Mathematics\\
         Texas A\&M University\\
         College Station\\
         Texas \ 77843\\
         USA}
 \email{sottile@tamu.edu}
\urladdr{franksottile.github.io/}
\thanks{Research supported in part by NSF grant DMS-1501370.}
\thanks{2017 CIMPA Research School in Ibadan, Nigeria supported in part by CIMPA, IMU,
        ICTP, Perimeter Institute, and the Fields Institute. 
        Web page at {\tt franksottile.github.io/conferences/CIMPA17}.}
\begin{document}

\begin{abstract} 
 These notes are based on, and significantly extend, Frank Sottile's short course of four lectures at the 
 CIMPA school on Combinatorial and  Computational Algebraic Geometry in Ibadan, Nigeria that took place 12--23 June
 2017. 
\end{abstract}

\maketitle

%
%
%

\tableofcontents

Toric varieties are perhaps the most accessible class of algebraic varieties.
They often arise as varieties parameterized by monomials, and their structure may be completely understood through
objects from geometric combinatorics.
While accessible and understandable, the class of toric varieties is also rich enough to illustrate many properties of
algebraic varieties.
Toric varieties are also ubiquitous in applications of mathematics, from tensors to statistical models to geometric
modeling to solving systems of equations, and they are important to other branches of mathematics such as geometric
combinatorics and tropical geometry.

For additional reference, see~\cite{CLS,Ewald,Fulton,Sottile} (the last is freely
accessible and covers some material from the perspective of real toric varieties).
For an accessible background on algebraic geometry and Gr\"obner bases, we recommend~\cite{CLO}, which is a classic and won
the American Mathematical Society's Leroy P. Steele Prize for Exposition in 2016~\cite{Steele}.

\subsection*{Notation and a note about our field}
We write $\C$ for the complex numbers, $\R$ for the real numbers, $\Q$ for the rational numbers, $\Z$ for the integers, and $\N$ for the
natural numbers (nonnegative integers).
While we describe complex toric varieties, the description holds verbatim for any field as 
toric varieties are naturally schemes over $\mbox{spec}(\Z)$.
When the ambient field $\Bbbk$ is not algebraically closed, there is an attractive theory of {\sl arithmetic toric varieties}~\cite{ATV}.

\section{Affine Toric Varieties}\label{S:one}
Recall that every finitely generated free abelian group $G$ is isomorphic to $\Z^m$ for some positive integer $m$ called
the \demph{rank} of $G$, and the isomorphism is equivalent to choosing a basis for $G$.

Write \defcolor{$\C^*$} for the group of nonzero complex numbers and \defcolor{$(\C^*)^n$} for the complex
torus of invertible diagonal $n\times n$ complex matrices, equivalently, of ordered $n$-tuples of nonzero complex numbers. 
The free abelian group $\Z^n$ of rank $n$ is associated to $(\C^*)^n$ in two distinct ways.
It is isomorphic to the lattice of one-parameter subgroups $\defcolor{\Hom_{g}}(\C^*,(\C^*)^n)$ of group homomorphisms
from $\C^*$ to $(\C^*)^n$. 
These are also called \demph{cocharacters}.
An integer vector $w=(w_1,\dotsc,w_n)\in\Z^n$ gives the map which
sends $t\in\C^*$ to the diagonal matrix $\defcolor{t^w}:=\diag (t^{w_1},\dotsc,t^ {w_n})\in(\C^*)^n$.
The group of characters, $\Hom_{g}((\C^*)^n,\C^*)$, equivalently of \demph{Laurent monomials}, is also isomorphic to $\Z^n$.
Here, an integer vector $a=(a_1,\dotsc,a_n)^T\in\Z^n$ gives the Laurent monomial 
$\defcolor{x^a}:=x_1^{a_1}\dotsb x_n^{a_n}$, which is also a group homomorphism $(\C^*)^n\ni x\mapsto x^a \in\C^*$, where
$x=\diag(x_1,\dotsc,x_n)$.

This ambiguity in the two roles for $\Z^n$ is resolved by writing \defcolor{$N$} for the cocharacters
and \defcolor{$M$} for the characters.
When expressed as integer vectors, elements of $N$ will be row vectors and those of $M$ column vectors.
Applying a character $a\in M$ to a cocharacter $w\in N$ gives a character of $\C^*$, which is an integer,
well-defined up to sign.
A standard choice gives the standard Euclidean pairing $N\otimes M\to\Z$, which we may see by computing
\[
   (t^w)^a\ =\  (t^{w_1})^{a_1}\dotsb (t^{w_n})^{a_n}\ =\ 
   t^{w_1a_1+\dotsb+w_na_n}\ =\  t^{w\cdot a}\,.
\]
The coordinate ring of $(\C^*)^n$ is the ring $\C[x_1,x_1^{-1},\dotsc,x_n,x_n^{-1}]$ of Laurent polynomials.
This is also the group algebra $\C[M]$ and we write \defcolor{$\C[x^{\pm}]$} for this Laurent ring.

\subsection{Affine Toric Varieties}

Let $\defcolor{\calA}\subset M\simeq\Z^n$ be a finite subset of monomials/characters.
It is convenient to represent $\calA$ as the set of column vectors of an integer matrix with $n$ rows.
We will also write $\calA$ for this matrix.
We will use this set $\calA$ to index coordinates, variables, etc.
For example $(\C^*)^\calA$ is the set of functions from $\calA$ to $\C^*$.
It is the algebraic torus $(\C^*)^{|\calA|}$ whose coordinates are indexed by the elements of $\calA$.
Likewise $\C^\calA$ is the vector space of functions from $\calA$ to $\C$.
It has coordinates $(z_a\mid a\in A)$.
If $\calA$ is represented by the matrix $\MyVect{0&0&1&1}{0&1&0&1}$,
then $\C^\calA\simeq\C^4$ has coordinates $z_{\MySmVec{0}{0}}$, $z_{\MySmVec{0}{1}}$, $z_{\MySmVec{1}{0}}$, $z_{\MySmVec{1}{1}}$.

This set $\calA$ may be used to define a map $\defcolor{\varphi_\calA}\colon(\C^*)^n\to\C^\calA$, where
 \begin{equation}\label{Eq:varphiA}
   \varphi_\calA(x)\ :=\ (x^a \mid a\in\calA)\,.
 \end{equation}
In the example where $\calA$ is represented by the matrix $\MyVect{0&1&0&1}{0&0&1&1}$, for $(x,y)\in(\C^*)^2$,
$\varphi_\calA(x,y)=(1,x,y,xy)\in\C^\calA$. 
Notice that the map $\varphi_\calA$~\eqref{Eq:varphiA} is a group homomorphism
$\varphi_\calA\colon(\C^*)^n\to(\C^*)^\calA$ followed by the inclusion  $(\C^*)^\calA\hookrightarrow\C^\calA$.
The Zariski closure of the image $\varphi_\calA((\C^*)^n)$ in $\C^\calA$ is the \demph{affine toric variety $X_\calA$}.
We deduce two characterizations of affine toric varieties from this definition.
Affine toric varieties are varieties that arise as the closure in $\C^m$ of a subtorus of $(\C^*)^m$, and affine toric
varieties are varieties that are parameterized by monomials.
Note that the torus $(\C^*)^n$ acts on $X_\calA$ with a dense orbit and this action extends to the ambient affine space
$\C^\calA$. 

If the subgroup \defcolor{$\Z\calA$} of $M$ generated by $\calA$ is a proper subgroup, then the homomorphism
$\varphi_\calA$ has a nontrivial kernel $\defcolor{T}:=\ker\varphi_\calA$.
In this case, $\varphi_\calA$~\eqref{Eq:varphiA} induces an injective map on $(\C^*)^n/T\to\C^\calA$.
In Exercise~\ref{Ex:injective}, you are asked to verify these claims and identify the kernel.
We have that $\Z\calA$ is the lattice of characters of $(\C^*)^n/T$ and so 
$\dim X_\calA=\dim (\C^*)^n/T=\rank \Z\calA$.

\begin{example}\label{Ex:AffineTV}
 Suppose that $n=1$ and $\calA=\{2,3\}\subset\Z$.
 For $s\in\C^*$,  $\varphi_\calA(s)=(s^2,s^3)\in\C^2$.
 The closure of $\varphi_\calA((\C^*)^n)$ is the cuspidal cubic, $\calV(y^2-x^3)$, where $\C^2$ has coordinates
 $(x,y)$:
\[
  \begin{picture}(81,81)
    \put(25,74){$y$}
    \put(75,45.5){$x$}
    \put(0,0){\includegraphics{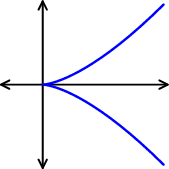}}
  \end{picture}
\]
 Since $\Z\calA=\Z$, $\varphi_\calA$ is injective.
 This may also be seen as if $(x,y)=\varphi_\calA(s)$, then $s=y/x$.

 Suppose that $k,m\geq 1$ are integers.
 Let $\be_1,\dotsc,\be_k$ and $\bff_1,\dotsc,\bff_m$ be the standard unit basis vectors for $\Z^k$
 and $\Z^m$, respectively, and set 
\[
   \calA\ :=\ 
    \{ \be_i+\bff_j\mid i=1,\dotsc,k\mbox{ and }j=1,\dotsc,m\}\subset\Z^k\times\Z^m\,,
\]
 which has $mn$ elements.
 The map $\varphi_\calA\colon\C^k\times\C^m\to\C^{km}$ is 
\[
   (x_1,\dotsc,x_k\,,\,y_1,\dotsc,y_m)\ \longmapsto\ (x_iy_j\mid  i=1,\dotsc,k\mbox{ and }j=1,\dotsc,m)\,.
\]
 If $\C^{km}$ is identified with the space of $k\times m$ matrices, this map is $(x,y)\mapsto x y^T$, and 
 thus $X_\calA$ is the space of $k\times m$ matrices of rank at most 1.

 Finally, suppose that $d\geq 1$ is an integer.
 The $d$th \demph{Veronese map} $\varphi_\calA \colon \C^n\to\C^{\binom{d+n}{n}}$, is when $\calA$ is the set of all
 exponent vectors in $\N^n$ of degree at most $d$. 
 When $n=1$ and $d=3$, we have $\calA=\{0,1,2,3\}\subset\Z$ and $\varphi_\calA(x)=(1,x,x^2,x^3)$.
 Ignoring the first coordinate which is constant, $X_\calA$ is the moment (rational normal) curve in $\C^3$.
 \qed
\end{example}

\subsection{Toric Ideals}

The ideal \defcolor{$I_\calA$} of $X_\calA$ is a \demph{toric ideal}.
It is an ideal of the coordinate ring \defcolor{$\C[z_a\mid a\in\calA]$} of the affine space $\C^\calA$.
To understand $I_\calA$, consider the pullback map corresponding to $\varphi_\calA$ on coordinate rings,
 \begin{eqnarray*}
   \defcolor{\varphi^*_\calA}\ \colon\ \C[z_a\mid a\in\calA] &\longrightarrow& \C[x_1,x_1^{-1},\dotsc,x_n,x_n^{-1}]\\
       z_a\mbox{\qquad} &\longmapsto & \mbox{\qquad} x^a\ .
 \end{eqnarray*}
The toric ideal $I_\calA$ is the kernel of $\varphi^*_\calA$.
The exponent of a monomial $z^u$ in $\C[z_a\mid a\in\calA]$ is the vector $u=(u_a\mid a\in\calA)\in\N^\calA$, and the
image of $z^u$ under $\varphi^*_\calA$ is 
\[
    \varphi^*_\calA(z^u)\ =\ \prod_{a\in\calA} (x^a)^{u_a}\ =\  x^{\sum a u_a}\,.
\]
Let us write the sum $\sum_{a\in\calA} a u_a$ in this exponent as \defcolor{$\calA u$}.
When $\calA$ is represented by an integer matrix $\calA$, this is the usual matrix-vector product.
Observe that the kernel $I_\calA$ of $\varphi^*_\calA$ contains the following set of binomials
 \begin{equation}
 \label{Eq:toric_binomials}
   \{  z^u - z^v \mid \calA u = \calA v\}\,. 
 \end{equation}
Suppose that $\calA$ is represented by the matrix $\MyVect{0&1&2&3&4}{4&3&2&1&0}$.
If $u=(0,1,1,1,0)^T$ and $v=(1,0,1,0,1)^T$, then $\calA u=\calA v$, which gives the binomial in $I_\calA$,
\[
   z_{\MySmVec{1}{3}}z_{\MySmVec{2}{2}}z_{\MySmVec{3}{1}}
   \ -\ 
   z_{\MySmVec{0}{4}}z_{\MySmVec{2}{2}}z_{\MySmVec{4}{0}}\,.
\]

\begin{theorem}\label{Th:Lin_Span_Toric}
 The toric ideal $I_\calA$ is a prime ideal.
 As a complex vector space, it is spanned by the binomials~\eqref{Eq:toric_binomials}.
\end{theorem}

\begin{proof}
 The image of $\varphi^*_\calA$ is the subalgebra of $\C[x^{\pm}]$ generated by the monomials $\{x^a\mid a\in\calA\}$. 
 Since $\C[x^{\pm}]$ is a domain, the kernel $I_\calA$ is a prime ideal.
 An equivalent way to see this is to note that $X_\calA$ is irreducible (hence its coordinate ring is a domain and
 its defining ideal is prime) as $X_\calA$ is the closure of the image of the irreducible variety $(\C^*)^n$ under the map
 $\varphi_\calA$. 

 For the second statement, let $\prec$ be any term order  on $\C[z_a\mid a\in\calA]$.
 Let $f\in I_\calA$.
 We may write $f$ as 
\[
   f\ =\ c_uz^u\ +\ \sum_{v\prec u} c_v z^v\qquad  c_u\neq 0\,,
\]
 so that $\ini_\prec(f)=c_uz^u$ is the initial term of $f$.
 Then
\[
  0\ =\ \varphi^*_\calA(f)\ =\  
   c_ux^{\calA u}\ +\ \sum_{v\prec u} c_v x^{\calA v}\ .
\]
 There is some $v\prec u$ with $\calA v=\calA u$, for otherwise the term $c_ux^{\calA u}$ is not canceled in
 $\varphi^*_\calA(f)$ and $\varphi^*_\calA(f)\neq 0$.
 Set $\overline{f}:= f-c_u(z^u-z^v)$.
 Then $\varphi^*_\calA(\overline{f})=0$ and $\ini_\prec(\overline{f})\prec\ini_\prec(f)$.

 If the leading term of $f$ were $\prec$-minimal in the initial ideal $\ini_\prec(I_\calA)$, then 
 $\overline{f}$ would be zero, and so $f$ is a scalar multiple of a binomial of the 
 form~\eqref{Eq:toric_binomials}.
 Suppose now that  $\ini_\prec f$ is not minimal in  $\ini_\prec(I_\calA)$ and that every polynomial in $I_\calA$ all of
 whose terms are $\prec$-less than the initial term of $f$ is a linear combination of binomials of the 
 form~\eqref{Eq:toric_binomials}.
 Then $\overline{f}$ is a linear combination of binomials of the 
 form~\eqref{Eq:toric_binomials}, which implies that $f$ is as well, by our induction hypothesis.
 This completes the proof.
\end{proof}

A \demph{monoid} has an associative binary operation and an identity element, but it does not necessarily have inverses.
(A \demph{semigroup} only has an associative binary operation, and not necessarily an identity.)
While many authors use the adjective semigroup when working with toric varieties, the identification below of maximal ideals with monoid
homomorphisms shows the inadequacy of that language.
Note that the complex numbers under multiplication forms a monoid and any group is also a monoid.
Define \defcolor{$\N\calA$} to be the submonoid of $M$ generated by $\calA$.
It consists of all linear combinations of elements of $\calA$ whose coefficients are natural numbers.
Write \defcolor{$\C[\N\calA]$} for the \demph{monoid algebra}, which is the set of complex-linear combinations of
elements of $\N\calA$.
It is also the set of Laurent polynomials whose exponents are from $\N\calA$.

\begin{corollary}\label{Co:CoordinateRing}
 The coordinate ring of the affine toric variety $X_\calA$ is $\C[\N\calA]$.
\end{corollary}

\begin{proof}
 The map $\sum_{a\in\calA} a n_a\mapsto \prod_{a\in\calA}(x^a)^{n_a}$ is a bijection between the submonoid $\N\calA$
 and the monoid of monomials $(x^a\mid a\in\calA)$ generated by $\calA$.
 In the proof of Theorem~\ref{Th:Lin_Span_Toric}, we identified the coordinate ring $\C[X_\calA]$ of $X_\calA$ with the
 subring of $\C[x^\pm]$ generated by the monomials $\{x^a\mid a\in\calA\}$, which is the monoid algebra $\C[\N\calA]$
 by the identification of $\N\calA$ with the monomials in $\C[X_\calA]$.
\end{proof}

Theorem~\ref{Th:Lin_Span_Toric} gives an infinite generating set for $I_\calA$.
We seek more economical generating sets.
Suppose that $\calA u=\calA v$ with $u,v\in\N^\calA$.
We define vectors $r,w^{\pm}$.
For $a\in\calA$, set
 \begin{eqnarray*}
  \defcolor{r_a} &:=& \min(u_a,v_a)\,,\\
  \defcolor{w^+_a} &:=& \max(u_a-v_a,0)\,,\quad\mbox{and}\\
  \defcolor{w^-_a} &:=& \max(v_a-u_a,0)\,.
 \end{eqnarray*}
Then $w^+,w^-\in\N^\calA$ and we have $u-v=w^+-w^-$ with $u=r+w^+$ and $v=r+w^-$, and  so
 \begin{equation}\label{Eq:Primitive_relation}
   z^r(z^{w^+} - z^{w^-})\ =\ z^u - z^v\ \in\ I_\calA\,,
 \end{equation}
with $z^r=\gcd\{z^u,z^v\}$. 
Note also that $\calA w^+=\calA w^-$ as $0=\calA(u-v)=\calA(w^+-w^-)$, so that $z^{w^+}-z^{w^-}\in I_\calA$.
In our example where $\calA=\MyVect{0&1&2&3&4}{4&3&2&1&0}$, we have $r=(0,0,1,0,0)^T$, $w^+=(0,1,0,1,0)^T$ and
$w^-=(1,0,0,0,1)^T$, and 
\[
   z_{\MySmVec{1}{3}}z_{\MySmVec{2}{2}}z_{\MySmVec{3}{1}}
   \ -\ 
   z_{\MySmVec{0}{4}}z_{\MySmVec{2}{2}}z_{\MySmVec{4}{0}}
  \ =\ 
   z_{\MySmVec{2}{2}}\bigl(z_{\MySmVec{1}{3}}z_{\MySmVec{3}{1}} \ -\  z_{\MySmVec{0}{4}}z_{\MySmVec{4}{0}}\bigr)
   \ \in\  I_\calA\,.
\]
For $u\in\Z^\calA$, let \defcolor{$u^+$} be the coordinatewise maximum of $u$ and the $0$-vector,
and let \defcolor{$u^-$} be the coordinatewise maximum of $-u$ and the $0$-vector.

\begin{corollary}\label{Co:GeneratorsFromKernel}
  $I_\calA=\langle z^{u^+}-z^{u^-}\mid \calA u=0\rangle$.
\end{corollary}

Thus $I_\calA$ is generated by binomials coming from the integer kernel of the matrix $\calA$. 

\begin{theorem}\label{Thm:binomial-friendly}
  Any reduced Gr\"obner basis of $I_\calA$ consists of binomials.
\end{theorem}

The point of this theorem is that Buchberger's algorithm is binomial-friendly.
That is, if $f$ and $g$ are binomials, then their $S$-polynomial is again a binomial, and the
reduction of one binomial by another is again a binomial.
You are asked to prove Theorem~\ref{Thm:binomial-friendly} in Exercise~\ref{Exer:binomial-friendly}.

It is an important problem to compute or to find relatively small Gr\"obner bases for toric ideals.
By Corollary~\ref{Co:GeneratorsFromKernel}, these are
given by special subsets of the integer kernel $\{u\in\Z^\calA\mid \calA u = 0\}$ of $\calA$.
For example, a reduced Gr\"obner basis for the ideal of $k\times m$ matrices of rank 1 is given by
%
%
\[
    \{\underline{z_{a,b}z_{c,d}}-z_{a,d}z_{c,b}\mid 1\leq a<c\leq k\mbox{ and }1\leq b<d\leq m\}\,,
\]
where the term order is the degree reverse lexicographic order with the variables ordered by 
$z_{a,b}\prec z_{c,d}$ if $a<c$ or $a=c$ and $b>d$ (the leading term is underlined).
This is the set of all $2\times 2$ minors of the $k\times m$ matrix $(z_{ij})_{i=1,\dotsc,k}^{j=1,\dotsc,m}$ of indeterminates.

By Corollary~\ref{Co:CoordinateRing}, the coordinate ring of the toric variety $X_\calA$ is the
algebra $\C[\N\calA]\simeq \C[x^a\mid a\in\calA]$.
This subalgebra of the ring $\C[x^{\pm}]=\C[M]$ of Laurent polynomials is spanned by monomials.
Let us generalize this.
Given a finitely generated submonoid $S$ of $M$ (this is a subset of $M$ that contains $0$ and is closed under addition), write
$\defcolor{\C[S]}\subset\C[x^{\pm}]$ for the  \demph{monoid algebra} of $S$.
This is the set of all complex-linear combinations of elements of $S$, where the multiplication is distributive and induced
by the monoid operation of $S$.
Choosing a generating set $\calA$ of $S$, so that $S=\N\calA$, realizes $\C[S]$ as the coordinate ring of the affine
toric variety $X_\calA\subset\C^\calA$. 
Then the usual algebraic-geometry dictionary implies that
\[
   X_\calA\ =\ \spec(\C[\N\calA])\,.
\]
Thus $\spec(\C[S])$ is an affine toric variety without a preferred embedding into affine space.
Under the algebraic-geometry dictionary, the closed points $X_\calA(\C)$ of $X_\calA$ correspond to the maximal ideals
of $\C[\N\calA]$. 

There is a second perspective on these points, via monoid homomorphisms.
By Exercise~\ref{Exer:Monoid}, maximal ideals correspond to monoid homomorphisms \defcolor{$\Hom_{m}(S,\C)$}, from
$S$ to $\C$ (additive on $S$ and multiplicative on $\C$).
An element $\frakm$ of $\Hom_{m}(S,\C)$ is a function $\frakm\colon S\to\C$ such that $\frakm(0)=1$ and for $a,b\in S$,
$\frakm(a+b)=\frakm(a)\cdot\frakm(b)$.

The map $\varphi_\calA$~\eqref{Eq:varphiA} restricts to the real torus 
$\defcolor{(\R^*)^n}\subset(\C^*)^n$ and gives a map $\varphi_\calA\colon(\R^*)^n\to\R^\calA$.
The closure of the image is the \demph{real affine toric variety $X_\calA(\R)$}.
If we write \defcolor{$\R_>$} for the positive real numbers and \defcolor{$\R_\geq$} for the nonnegative real numbers,
then we may also restrict $\varphi_\calA$ to $\R^n_>$.
Note that $\varphi_\calA(\R^n_>)\subset\R_\geq^\calA$, the positive orthant in $\R^\calA$.
The \demph{nonnegative affine toric variety $X_{\calA,\geq}$} is the closure of the image.
We have the following maps
 \begin{equation}\label{Eq:AffineFunctoriality}
   X_{\calA,\geq}\ \lhra\ X_\calA(\R)\ \lhra\ X_\calA(\C)\ \relbar\joinrel\twoheadrightarrow\ X_{\calA,\geq}\ .
 \end{equation}
These are induced by the maps 
 \begin{equation}\label{Eq:semiGroupMaps}
   \R_\geq\ \lhra\ \R\ \lhra\ \C\ \relbar\joinrel\twoheadrightarrow\ \R_\geq\,,
 \end{equation}
with the last map $z\mapsto|z|$.
The maps~\eqref{Eq:AffineFunctoriality} also come from the 
identification of affine toric varieties with monoid homomorphisms and the sequence of maps
of monoids~\eqref{Eq:semiGroupMaps}.
As the composition of these monoid maps is the identity, the composition~\eqref{Eq:AffineFunctoriality} is also the  
identity.


\subsection*{Exercises}

\begin{enumerate}[1.] 
 \item \label{Ex:injective}
       Show that the monomial map $\varphi_\calA\colon(\C^*)^n\to(\C^*)^\calA$~\eqref{Eq:varphiA} is injective if and only
       if $\calA$ generates $M$.
       When $\calA$ does not generate $M$, identify the kernel of $\varphi_\calA$.

 \item      
   Let $n\in\N$ be a natural number.
   Describe generators of the toric ideal $I_\calA$ for the point set $\calA=\left(\begin{matrix}0&1&2&3&\cdots&n\end{matrix}\right)$.
   Do the same for the point set 
   $\calA=\MyVect{0&1&2&\cdots&n}{n&\cdots&2&1&0}$.

 \item \Maroon{{\sf Breakfast Problem.}} Suppose that $\calA\subset\Z^4$ 
    is represented by the $4\times 8$ matrix
   \[
      \left(\begin{matrix} 
      1&1&1&1&1&1&1&1\\
      1&1&1&1&0&0&0&0\\
      1&1&0&0&1&1&0&0\\
      1&0&1&0&1&0&1&0 \end{matrix}\right)\ .
   \]
    Find a set of nine linearly independent generators of the toric ideal $I_\calA$.

 \item Let $\calA\subset\Z^6=(\Z^2)^3$ be represented by  the $6\times 8$ matrix
   \[
      \left(\begin{matrix} 
      1&1&1&1&0&0&0&0\\
      0&0&0&0&1&1&1&1\\
      1&1&0&0&1&1&0&0\\
      0&0&1&1&0&0&1&1\\
      1&0&1&0&1&0&1&0\\
      0&1&0&1&0&1&0&1 \end{matrix}\right)\ .
   \]
    Find linearly independent generators of the toric ideal $I_\calA$.
    Hint: The even (respectively odd) numbered rows give the vertices of the $3$-cube.
    What is $X_\calA$?
    For this, consider the map $\varphi_\calA$ where the rows correspond to the variables $x_0,x_1,y_0,y_1,z_0,z_1$ and the
    columns to the variables $p_{000},p_{100},\dotsc,p_{111}$.

 \item  Show that the collection of $2\times 2$ minors of the  $k\times m$ matrix
   $(z_{ij})_{i=1,\dotsc,k}^{j=1,\dotsc,m}$ of indeterminates forms a reduced Gr\"obner basis for the toric ideal of the
   variety $X_\calA$ of matrices of rank 1, where the term order is degree reverse lexicographic with the variables
   ordered by $z_{a,b}\prec z_{c,d}$ if $a<c$ or $a=c$ and $b>d$.
   What about other term orders?

 \item \label{Exer:ToricIdeals}
   Identify a generating set and a reduced Gr\"obner basis for the toric ideal $I_\calA$ given by as many of the
   following six subsets of $\Z^2$ as you can. 
   The origin is at the lower left of each figure.
   You may find computer software useful.
\[
  \includegraphics{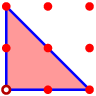}\qquad
  \includegraphics{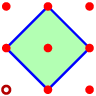}\qquad
  \includegraphics{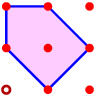}\qquad
  \includegraphics{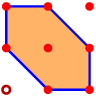}\qquad
  \includegraphics{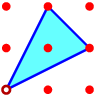}\qquad
  \raisebox{-20.5pt}{\includegraphics{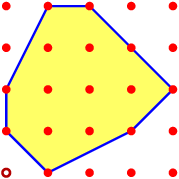}}
\]

 \item  \label{Exer:binomial-friendly}
  Work out the details of the suggested proof of Theorem~\ref{Thm:binomial-friendly}.
  Explain why `reduced' is necessary for the conclusion.

 \item
   A submonoid $S\subset M$ is \demph{saturated} if for any $a\in M$ and $k\in \N$ with $k\neq 0$, if $ka\in S$, then $a\in S$.
   Show that if $\C[S]$ is normal, then $S$ is saturated.
   Can you prove the converse of this statement?

 \item \label{Exer:Monoid}
       Let $S$ be a finitely generated submonoid of $M$.
       Show that every maximal ideal $\frakm$ of $\C[S]$ restricts to a monoid homomorphism 
       from $S$ to $\C$, and vice-versa.
       Hint: use that maximal ideals correspond to algebra maps $\C[S]\to\C$.

 \item  For $\calA\subset M$ finite, show that 
      $X_\calA(\R)=\Hom_{m}(\N\calA,\R)$ and 
      $X_{\calA,\geq}=\Hom_{m}(\N\calA,\R_\geq)$.

 \item Use the embedding $x\mapsto (x,x^{-1})$ of the torus $\C^*$ into $\C^2$, or any other method, to identify the
   coordinate ring of the torus $\C^*$ with $\C[x,x^{-1}]$.
   Deduce that the coordinate ring of $(\C^*)^n$ may be identified with $\C[x_1,x_1^{-1},\dotsc,x_n,x_n^{-1}]$.
   Show that this is the complex group ring $\C[M]$ of the free abelian group $M$ of characters of the torus $(\C^*)^n$.
   Harder: Can you relate algebraic structures of $\C[M]$ to the group structure on $(\C^*)^n$?
   That is, what do the product, inverse, and identity element of  $(\C^*)^n$ correspond to on $\C[M]$ and $M$.
\end{enumerate}

\section{Toric Projective Varieties and Solving Equations}\label{S:two}

Some affine toric varieties may be considered to be projective varieties---this is when they are stable under the
multiplication by scalars on their ambient vector space.
In this case, they have particularly attractive properties.
Such projective toric varieties also provide a means to prove one of the signature results related to toric varieties;
Kushnirenko's Theorem about the number of solutions to a system of sparse polynomial equations.

\subsection{Toric Varieties in Projective Space}\label{SS:ProjectiveTV}
Projective space $\defcolor{\P^m}=\defcolor{\P(\C^{m+1})}$ is the set of one-dimensional linear subspaces of $\C^{m+1}$.
Since a one-dimensional linear subspace \defcolor{$\ell$} is generated by any nonzero point of $\ell$, and scalar
multiplication by elements of $\C^*$ acts simply transitively on these nonzero points of $\ell$, and freely on
$\C^{m+1}\smallsetminus\{0\}$, we may identify $\P^m$ with the quotient $(\C^{m+1}\smallsetminus\{0\})/\C^*$.
Projective space is equipped with \demph{homogeneous coordinates} $[z_0:\dotsb:z_m]$, were we identify 
$[z_0:z_1:\dotsb:z_m]$ with $[tz_0:tz_1:\dotsb:tz_m]$ for any nonzero scalar $t\in\C^*$.
An affine variety $X\subset\C^{m+1}$ corresponds to a projective variety in $\P(\C^{m+1})=\P^m$ when $X$ is
homogeneous under the $\C^*$-action on $\C^{m+1}$ given by scalar multiplication.
We will call such an affine toric variety $X\subset\C^{m+1}$ the \demph{affine cone} over the corresponding projective toric
variety $X\subset\P^{m}$.

We claim that an affine toric variety $X_\calA\subset\C^\calA$ is homogeneous when the set $\calA$ lies on an affine
hyperplane. 
By this, we mean that there is some $w\in N=\Z^n$ with
\[
   w\cdot  a \ =\ w\cdot  b \qquad \mbox{ for all } a,b\in\calA\,,
\]
and this common value \defcolor{$c$} is nonzero.
(Here, an affine hyperplane does not contain the origin.)
Then, under the composition of the cocharacter, $t\mapsto t^w$ and the map $\varphi_\calA$, $t\in\C^*$ acts as
multiplication by the scalar $t^c$ on $\C^\calA$ as  $\varphi_\calA(t^w) = (t^{w\cdot  a}=t^c\mid a\in\calA)$, and thus
$X_\calA\subset\C^\calA$ is homogeneous under scalar multiplication.
For another way to see this, suppose that $u,v\in\N^\calA$ are integer vectors with $\calA u = \calA v$.
Then $w\cdot\calA u = w\cdot\calA v$, which implies that 
\[
   c\sum_{a\in \calA} u_a\ =\  c\sum_{a\in \calA} v_a\,,
\]
and thus $z^u-z^v\in I_\calA$ is homogeneous.
By Theorem~\ref{Th:Lin_Span_Toric}, we obtain the following.

\begin{corollary}
  If $\calA$ lies on an affine hyperplane, then $I_\calA$ is a homogeneous ideal and $X_\calA$ is a projective
  subvariety of\/ $\defcolor{\P^\calA}:=\P(\C^\calA)$.
\end{corollary}

\begin{example}\label{Ex:TwCubic}
 Suppose that $\calA$ is represented by the matrix
 $\MyVect{3&2&1&0}{0&1&2&3}$, which are the points $a$ of $\N^2$ where $(1,1)\cdot a=3$.
\begin{figure}[htb]
   \begin{picture}(60,60)
    \put(0,0){\includegraphics{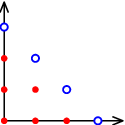}}
    \put(42,43){$\calA$}
    \put(40,47){\vector(-1,0){34}}
    \put(40,43.5){\vector(-2,-1){19}}
    \put(43.5,40){\vector(-1,-2){9.5}}
    \put(47,40){\vector(0,-1){34}}
  \end{picture}
  \qquad\qquad
  \raisebox{-10pt}{\begin{picture}(98,80)(-17,0)
    \put(0,0){\includegraphics{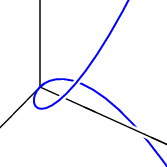}}
    \put(-17,27){$z_{\MySmVec{2}{1}}$} \put(76,19){$z_{\MySmVec{1}{2}}$} \put(21,74){$z_{\MySmVec{0}{3}}$}
  \end{picture}}

\caption{Exponents $\calA$ and the twisted cubic.}
 \label{F:calA_TwistedCubic}
\end{figure}
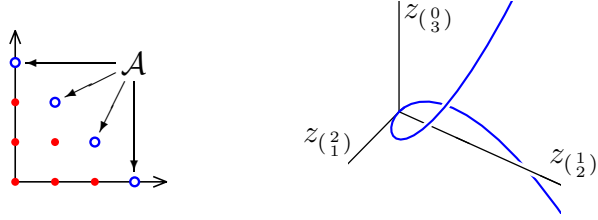
Then $\varphi_\calA(x,y)=(x^3,x^2y,xy^2,y^3)\in\C^\calA\simeq\C^4$, and the closure $X_\calA$ of its image in
$\P^\calA=\P^3$ is the twisted cubic.
If $[z_{\MySmVec{3}{0}}:z_{\MySmVec{2}{1}}:z_{\MySmVec{1}{2}}:z_{\MySmVec{0}{3}}]$ are the coordinates of $\P^\calA$, then
the homogeneous toric ideal $I_\calA$ is generated by  
 \begin{equation}\label{Eq:TwistedCubicEquations}
   z_{\MySmVec{3}{0}}z_{\MySmVec{1}{2}} - z_{\MySmVec{2}{1}}^2\,,\ \ 
   z_{\MySmVec{3}{0}}z_{\MySmVec{0}{3}} - z_{\MySmVec{2}{1}}z_{\MySmVec{1}{2}}\,,\ \mbox{ and }\ 
   z_{\MySmVec{2}{1}}z_{\MySmVec{0}{3}} - z_{\MySmVec{1}{2}}^2\,,
 \end{equation}
which correspond to the vectors $(1,-2,1,0)^T$, $(1,-1,-1,1)^T$, and $(0,1,-2,1)^T$ in $\ker\calA$, which are also the
primitive relations among the elements of $\calA$,  
\[
   \begin{pmatrix}3\\0\end{pmatrix} + \begin{pmatrix}1\\2\end{pmatrix}
     \ =\  
    2 \begin{pmatrix}2\\1\end{pmatrix},
   \quad
   \begin{pmatrix}3\\0\end{pmatrix} + \begin{pmatrix}0\\3\end{pmatrix}
      \ =\ 
   \begin{pmatrix}2\\1\end{pmatrix} + \begin{pmatrix}1\\2\end{pmatrix},
   \quad\mbox{ and }\quad
   \begin{pmatrix}2\\1\end{pmatrix} + \begin{pmatrix}0\\3\end{pmatrix}
      \ =\ 
   2\begin{pmatrix}1\\2\end{pmatrix}.
\]
Here, $\Z\calA$ is a full rank sublattice of index 3 in $\Z^2$, which you are asked to show in Exercise~\ref{Exer:TwCubic}.
Also, the kernel of $\varphi_\calA$ is $\{ \MyVect{1}{1},\MyVect{\zeta}{\zeta},\MyVect{\zeta^2}{\zeta^2}\}$, where 
$\defcolor{\zeta}:=e^{\frac{2\pi\sqrt{-1}}{3}}$ is a cube root of 1, and thus we have that 
$\ker \varphi_\calA \simeq \Hom_{g}(\Z^2/\Z\calA,\C^*)$.
Choosing $\MyVect{3}{0}$ and $\MyVect{-1}{\tsp 1}=\MyVect{2}{1}-\MyVect{3}{0}$ as a basis for $\Z\calA$ (and identifying it
with $\Z^2$), the set $\calA$ becomes the columns of the matrix
$(\begin{smallmatrix}1&1&1&1\\0&1&2&3\end{smallmatrix})$, which we draw with the 
first coordinate vertical.
\[
  \begin{picture}(110,42)(-38,0)
   \put(0,0){\includegraphics{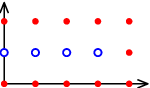}}
   \put(-38,14){$\calA$} \put(-27,17){\vector(1,0){25}}
  \end{picture}
\]
This is the set $\calA$ for the affine rational normal curve of Example~\ref{Ex:AffineTV} lifted to an affine hyperplane
in $\Z^2$ by prepending a new first coordinate of 1 to each element of $\calA$.\qed
\end{example}

Let $\calA\subset\Z^n$ be a finite set.
Its \demph{lift}, $\defcolor{\calA^+}\subset\Z^{1+n}$, is the set 
\[
   \defcolor{\calA^+}\ :=\ \{ (1,a) \mid a\in\calA\}\,,
\]
which lies on an affine hyperplane in $\Z^{1+n}$.
The map $\varphi_{\calA^+}\colon\C^*\times(\C^*)^n\to\P^\calA$ is given by
 \begin{equation}\label{Eq:liftMap}
  \varphi_{\calA^+}(t,x)\ =\ [tx^a\mid a\in\calA]\,.
 \end{equation}
Observe that the set of differences $\{ a-b\mid a,b\in\calA\}$ spans $\Z^n$ if and only if the set $\calA^+$ of vectors
spans $\Z^{1+n}$. 

\begin{example}\label{Ex:liftedHex}
 Suppose that $\calA$ is represented by the matrix
 $\MyVect{0&1&1&0&-1&-1&\tsp 0}{0&0&1&1&\tsp 0&-1&-1}$.
 Then $\calA$ consists of the integer points of the hexagon in Figure~\ref{F:lift}.
 Figure~\ref{F:lift} also shows the lift of this hexagon, where the first coordinate is
 vertical in the lift.\qed
\end{example}
\begin{figure}[htb]
  \includegraphics{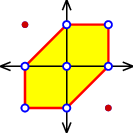} \qquad \includegraphics{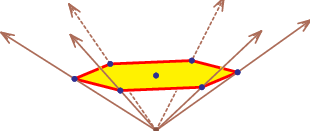}
 \caption{The hexagon and its lift.}
 \label{F:lift}
\end{figure}

In  Exercise~\ref{Exer:Lift}, you are asked to show that any finite set $\calA\subset\Z^{1+n}$ lying on an affine
hyperplane has the form $\calB^+$ in appropriate coordinates for $\Z\calA$.
Exercise~\ref{Exer:Row-span} gives another (equivalent) characterization of a set $\calA$ lying on an affine
hyperplane.

We turn to a geometric description of the generators of a homogeneous toric ideal $I_\calA$.
A sum  $\sum_{a\in\calA} a\lambda_a$ where $\sum_{a\in\calA} \lambda_a=1$ and $0\leq \lambda_a$ is a 
\demph{convex combination} of the points of $\calA$.
The \demph{convex hull} of $\calA\subset\Z^{n}$ is the set of all convex combinations of the points of $\calA$,
\[
   \defcolor{\conv(\calA)}\ :=\ 
    \Bigl\{ \sum_{a\in\calA} a\lambda_a \ \Bigl\vert\   \sum_{a\in\calA} \lambda_a=1
        \mbox{ and } 0\leq \lambda_a \mbox{ for all }a\in\calA\Bigr\}\,.
\]
This convex hull is a polytope with integer vertices (a \demph{lattice polytope}), and its vertices are a subset of
$\calA$. 
Lattice polygons were depicted in Exercise~\ref{Exer:ToricIdeals} of Section~\ref{S:one}
and in Figure~\ref{F:lift}.
Below are a lattice simplex
($\calA$ corresponds to the matrix $\MySmTVec{0&1&0&0}{0&0&1&0}{0&0&0&1}$),
a lattice cube 
($\calA$ corresponds to the matrix $\MySmTVec{0&1&0&0&1&0&0&1}{0&0&1&0&1&0&1&1}{0&0&0&1&0&0&1&1}$),
and a lattice octahedron
($\calA$ corresponds to the matrix $\MySmTVec{0&1&-1&0&0&0&0}{0&0&0&1&-1&0&0}{0&0&0&0&0&1&-1}$).
 \begin{equation}\label{Eq:SomePolytopes}
    \raisebox{-30pt}{\includegraphics[height=60pt]{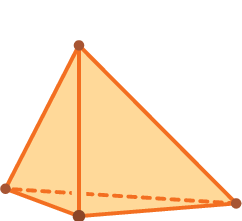}} \qquad\qquad
   \raisebox{-30pt}{\includegraphics[height=60pt]{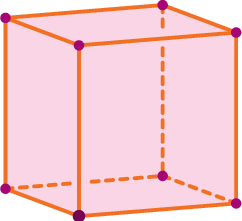}}        \qquad\qquad
  \raisebox{-48pt}{\includegraphics[height=95pt]{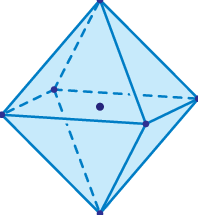}}
 \end{equation}

Let $\calA\subset\Z^{1+n}$ lie on an affine hyperplane.
Suppose that $u,v\in\N^\calA$ are nonzero vectors with $u\neq v$ and $\calA u = \calA v$ so that $z^u-z^v$ is a binomial in
$I_\calA$. 
By our assumption on $\calA$, the toric ideal $I_\calA$ is homogeneous, so that $\deg z^u = \deg z^v$.
Let $d:=\sum_a u_a=\sum_a v_a$ be this degree.
Writing $\defcolor{\lambda_a}:=\frac{1}{d}u_a$ and $\defcolor{\mu_a}:=\frac{1}{d}v_a$ for $a\in\calA$, we have 
\[
   \sum_{a\in\calA} a\lambda_a \ =\ \sum_{a\in\calA} a \mu_a \,.
\]
As $\lambda_a,\mu_a\geq 0$ are rational numbers and $\sum_a \lambda_a=\sum_a\mu_a=1$, 
this is a point in $\conv(\calA)$ having two distinct representations as a rational convex combination of
the points of $\calA$. 

Suppose that $\calA w=0$.
Then $z^{w^+}-z^{w^-}$ is a generator for $I_\calA$, by Corollary~\ref{Co:GeneratorsFromKernel}.
Note that $\supp(w^+)$ is disjoint from $\supp(w^-)$.
(Here, the support \defcolor{$\supp(v)$}
of a vector $v$ is the set of indices of nonzero coordinates.)
Then the above construction (applied to $w^+$ and $w^-$) gives
\[
   \sum_{a\in\supp(w^+)} a\lambda_a \ =\ \sum_{a\in\supp(w^-)} a \mu_a \,,
\]
which is a rational point common to the convex hulls of two disjoint subsets of $\calA$.

\begin{example}
 Suppose that $\calA$ is represented by the matrix
 $\MyVect{0&1&1&2&2}{1&0&2&0&1}$.
 The point $(\frac{3}{2},\frac{1}{2})^T$ lies in the convex hull of two disjoint subsets of $\calA$,
 $\MyVect{1&2}{0&1}$ and  $\MyVect{0&1&2}{1&2&0}$, respectively.
\[
  \begin{picture}(235,72)
    \put(0,0){\includegraphics{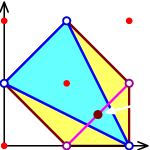}}
    \put(91,28){\vector(-4,-1){40}}
    \put(93,26){$\MyVect{3/2}{1/2}=\tfrac{1}{2}\MyVect{1}{0}+\tfrac{1}{2}\MyVect{2}{1}$}
    \put(93, 8){${\color{white}\MyVect{3/2}{1/2}}=\tfrac{1}{6}\MyVect{0}{1}+\tfrac{1}{6}\MyVect{1}{2}
                                                 +\tfrac{2}{3}\MyVect{2}{0}$}
  \end{picture}
\]
These coincident convex combinations give the binomial 
$z_{\MySmVec{1}{0}}^3z_{\MySmVec{2}{1}}^3-z_{\MySmVec{0}{1}}z_{\MySmVec{1}{2}}z_{\MySmVec{2}{0}}^4$ in $I_{\calA^+}$.
\qed
\end{example}

We summarize this discussion.

\begin{proposition}\label{P:coincidentConvex}
 Suppose that $\calA$ lies on an affine hyperplane.
 Homogeneous generators $z^{w^+}-z^{w^-}$ of $I_\calA$ of Corollary~\ref{Co:GeneratorsFromKernel} correspond to
 rational points of $\conv(\calA)$ lying in the intersection of convex hulls of two disjoint subsets of $\calA$.
\end{proposition}

For a finite subset $\calA\subset M\simeq\Z^n$, a cocharacter $w\in N$ (or any vector in $\defcolor{N_\R}:=N\otimes_\Z\R$) 
gives a function on $\calA$, where $a\mapsto w\cdot  a$.
Write \defcolor{$h_\calA(w)$} for the maximum value this function takes on points of $\calA$.
The function $w\mapsto h_\calA(w)$ is the \demph{support function} of $\calA$.
The subset of $\calA$ where the function $a\mapsto w\cdot  a$ attains its maximum, 
 \begin{equation}\label{Eq:calA_w}
   \defcolor{\calA_w}\ :=\ \{ a\in\calA\mid w\cdot  a = h_\calA(w)\}\,,
 \end{equation}
is the face of $\calA$ \demph{exposed} by $w$.
When $\calA$ is the set of column vectors of $\MySmTVec{0&1&0&0&1&0&0&1}{0&0&1&0&1&0&1&1}{0&0&0&1&0&0&1&1}$, which are the
vertices of the lattice cube, the face exposed by the vector $(1,2,3)$ is the vertex  
$\MySmTVec{1}{1}{1}$, the face exposed by the vector $(1,-2,0)$ is the subset 
$\{\MySmTVec{1}{0}{1},\MySmTVec{1}{0}{0}\}$ that spans an edge of the cube, and the face exposed by the
vector $(0,0,-1)$ is the subset
$\{\MySmTVec{0}{0}{0},\MySmTVec{0}{1}{0},\MySmTVec{1}{0}{0},\MySmTVec{1}{1}{0}\}$, which spans the
downward-pointing facet.
\[
  \begin{picture}(111,99)(-17,-15)
    \put(0,0){\includegraphics[height=70pt]{figures/Cube.eps}}
    \put( 77,  3){$\MySmTVec{1}{0}{0}$}    \put( 77, 58){$\MySmTVec{1}{0}{1}$}
    \put( 53, 19){$\MySmTVec{1}{1}{0}$}    \put( 50, 75){$\MySmTVec{1}{1}{1}$}
    \put(-17,  8){$\MySmTVec{0}{1}{0}$}    \put(-17, 62){$\MySmTVec{0}{1}{1}$}
    \put( 16,-12){$\MySmTVec{0}{0}{0}$}    \put( 25, 44){$\MySmTVec{0}{0}{1}$}
  \end{picture}
\]

A \demph{face} of $\calA$  is any subset of this form.
These same notions of support function $h_P(w)$, face $F$ of $P$, and face $P_w$ of $P$ exposed by $w$, also hold for
any polytope $P$.
Each face $\calF$ of $\calA$ is the intersection of $\calA$ with a face $F$ of its convex hull $\conv(\calA)$, and 
$F=\conv(\calF)$.
The inclusion $\calF\subset\calA$ of subsets induces an inclusion of projective spaces
$\P^\calF\subset \P^\calA$, where $\P^\calF$ is identified with the coordinate
subspace $\{z\in\P^\calA\mid z_a=0 \mbox{ if } a\not\in\calF\}$ of $\P^\calA$. 
We state a relation between faces of $\calA$ and toric subvarieties of $X_\calA$ without proof.

\begin{lemma}\label{L:BoundaryToricVars}
 Let $X_\calA\subset\P^\calA$ be the projective toric variety given by finite set $\calA\subset\Z^{1+n}$ lying on an
 affine hyperplane.
 For any point $z\in X_\calA$, its support $\supp(z) = \{ a\in \calA \mid z_a\neq 0\}$  is a face of $\calA$.
 For every face $\calF$ of $\calA$, the intersection $X_\calA\cap\P^\calF$ is naturally identified with $X_\calF$.
\end{lemma}

There is much more relating the structure of the polytope $\conv(\calA)$ and the toric variety.
We state another such result without proof.
Consider the map $\defcolor{\mu_\calA}\colon \P^\calA\to \conv(\calA)$ given by
\[
   \P^\calA\ \ni\ z=[z_a\mid a\in\calA]\ \longmapsto\ 
    \frac{\sum_{a\in\calA} a|z_a| }{\sum_{a\in\calA} |z_a|}\ \in\ \conv(\calA)\,.
\]

\begin{lemma}\label{L:Birch}
 The map $\mu_\calA\colon X_{\calA^+}\to\conv(\calA)$ is surjective.
 The inverse image of a face $F$ of $\conv(\calA)$ is $X_\calF$, where $\calF=F\cap\calA$.
 The map $\mu_\calA$ remains surjective when restricted to $X_{\calA^+}(\R)=X_{\calA^+}\cap\P^\calA(\R)$, where
 $\P^\calA(\R)=\P(\R^\calA)$, and also to
\[
   X_{\calA^+}(\R_\geq)\ :=\ 
  \{x=[x_a\mid a\in\calA]\in X_{\calA^+} \mid x_a\geq 0\mbox{ for all }a\in\calA\}\,,
\]
 on which it is a homeomorphism.
 This map identifies $X_{\calA^+}(\R)$ with $2^n$ copies of $\conv(\calA)$ glued along facets. 
\end{lemma}

\subsection{Kushnirenko's Theorem}\label{S:Kushnirenko}
We turn to one of the most celebrated applications of toric varieties, understanding
the number of solutions to a system of polynomial equations.
A Laurent polynomial $f\in\C[x^\pm]$ is a finite linear combination of monomials.
That is, there are coefficients $c_a\in\C$ for $a\in \Z^n$ such that 
\[
   f\ =\ \sum_a c_a x^a\,,
\]
with at most finitely many coefficients $c_a$ nonzero.
The set $\calA$ of indices of nonzero coefficients is called the \demph{support} of $f$ and its convex
hull $\conv(\calA)$ is the \demph{Newton polytope} of $f$, which is a lattice polytope.
We consider the number of solutions in $(\C^*)^n$ to a system
 \begin{equation}
  \label{E3:PolySystem} 
    f_1(x)\ =\ f_2(x)\ =\ \dotsb\ =\ f_n(x)\ =\ 0
 \end{equation}
of (Laurent) polynomial equations, where each polynomial has the same support $\calA$.
The coefficients of a polynomial identify $\C^\calA$ with the set of polynomials whose support is a
subset of $\calA$, and $(\C^\calA)^n$ is identified with set of polynomial systems~\eqref{E3:PolySystem} with support
$\calA$. 
Kushnirenko~\cite{Kushnirenko} proved the following count for the number of solutions to a system of
polynomial equations~\eqref{E3:PolySystem} with support $\calA$.

\begin{theorem}[Kushnirenko]
\label{Th:Kushnirenko}
 A system~\eqref{E3:PolySystem} of $n$ polynomials in $n$ variables 
 with support $\calA$ has at most $n!\Vol(\conv(\calA))$ isolated
 solutions in $(\C^*)^n$, counted with multiplicity.
 There is a dense open subset of  $(\C^\calA)^n$ consisting of systems with support $\calA$ having exactly
 $n!\Vol(\conv(\calA))$ solutions in $(\C^*)^n$, each isolated and occurring with multiplicity one. 
\end{theorem}

Lemma~\ref{L:genericSystems} of Section~\ref{S:Bernstein} establishes the claim that there is an open set in $(\C^\calA)^n$
of systems with support $\calA$ all of which have the same number of isolated solutions and where each solution occurs with 
multiplicity one. 
Theorem~\ref{Th:SparseDiscriminant} describes the discriminant conditions that imply all solutions are isolated.
Namely, that for each $w\in\R^n$, the \demph{facial system}
 \begin{equation}\label{Eq:facial-system}
    f_{1,w}(x)\ =\ f_{2,w}(x)\ =\ \dotsb\ =\ f_{n,w}(x)\ =\ 0
 \end{equation}
 has no solutions in $(\C^\times)^n$, where, for a Laurent polynomial $f$ with support $\calA$, $f_w$ is the restriction
 of $f$ to the monomials in $\calA_w$.
 This is also the initial form $\ini_w f$ of $f$ with respect to the weighted partial term order $\prec_w$.
 This partial term order is defined in Exercise~\ref{Ex:weightedTermOrder}, where you are asked to prove the previous claim.

We use the projective toric variety $X_{\calA^+}$ to prove Kushnirenko's Theorem.
The map $\varphi_\calA$ parameterizes $X_{\calA^+}$, and we 
first understand when this parametrization is injective.
The \demph{affine span} of a set $\calA$ is
 \begin{equation}\label{Eq:Affine-span}
   \defcolor{\Aff \calA}\ :=\      \Bigl\{ \sum_{a\in\calA} a\lambda_a \mid  \sum_{a\in\calA} \lambda_a=1 \Bigr\}\,.
 \end{equation}
This differs from the convex hull in that the coefficients $\lambda_a$ may be negative.
When $\lambda_a\in\Z$, this is the \demph{integral affine span} \defcolor{$\Aff_\Z\calA$}.
For any $a\in\calA$, the affine span is the coset
 \begin{equation}\label{Eq:Affine-coset}
   \Aff \calA\ =\ a \ +\ \R\{ b-a\mid b\in\calA\}\,,
 \end{equation}
and the same (replacing $\Z$ for $\R$) gives the integral affine span.

The map $\varphi_\calA\colon(\C^*)^n\to\P^\calA$ is the restriction of $\varphi_{\calA^+}$~\eqref{Eq:liftMap} to the subtorus 
$\{1\}\times(\C^*)^n$ of the torus $\C^*\times(\C^*)^n$ where $t=1$.
In Exercise~\ref{Ex:Prinjective} you are asked to show that $\varphi_\calA$ is injective if and only if 
$\Aff_\Z \calA=\Z^n$.
Notice that if $0\in\calA$, then $\Aff_\Z \calA=\Z\calA$.
Since multiplying a Laurent polynomial by a monomial $x^a$ does not change its set of zeros, it is no loss of generality
to assume that $0\in\calA$, in which case $\varphi_\calA$ is injective if and only if $\Z\calA=\Z^n$.
As $0\in\calA$, one of the coordinates of $\varphi_\calA$ is 1, so its image lies in a standard affine
patch of $\P^\calA$.

We relate the projective toric variety $X_{\calA^+}$ to systems of polynomials with support $\calA$.
Given a (homogeneous) linear form $\Lambda$ on $\P^\calA$, 
\[
  \defcolor{\Lambda}\ =\ \sum_{a\in\calA} c_a z_a\,,
\]
its pullback $\varphi^*_{\calA}(\Lambda)$ 
along $\varphi_{\calA}$ is a polynomial with support $\calA$,
\[
   \varphi^*_{\calA}(\Lambda)\ =\ 
    \sum_{a\in\calA} c_a x^a\,.
\]
Consequently, a system of $n$ polynomials~\eqref{E3:PolySystem} with support $\calA$ is the pullback along $\varphi_\calA$
of a system of $n$ linear forms on $\P^\calA$. 
Note that a linear form $\Lambda$ on $\P^\calA$ defines a hyperplane $H\subset\P^\calA$ and $n$ general linear forms
define a linear subspace $L$ of codimension $n$.

\begin{lemma}\label{L3:linearSection}
 The solution set of a system of polynomials~\eqref{E3:PolySystem} with support $\calA$
 is the pullback $\varphi_{\calA}^{-1}(L)=\varphi_{\calA}^{-1}(L\cap\varphi_\calA({\C^*}^n))$
 of a linear section of $\varphi_{\calA}({\C^*}^n)$, where $L$ has codimension equal to the
 dimension of the linear span of the polynomials $f_i$.
\end{lemma}

\begin{example}\label{Ex3:cubic}
 Consider the polynomial system
 \begin{equation}\label{Eq3:ex_sparse}
   \Blue{f}\ :=\ \Blue{x^2y+2xy^2-1+xy}\ =\ 0
     \quad\mbox{and}\quad 
   \Magenta{g}\ :=\ \Magenta{x^2y-xy^2+2-xy}\ =\ 0\,.
 \end{equation}
 These polynomials define two plane curves which have one real point of intersection
 at  $(1.53277,-0.90655)$ and are displayed in Figure~\ref{F:sparse_curves}.
 \begin{figure}[htb]
   \begin{picture}(240,132) 
    \put(0,0){\includegraphics[height=132pt]{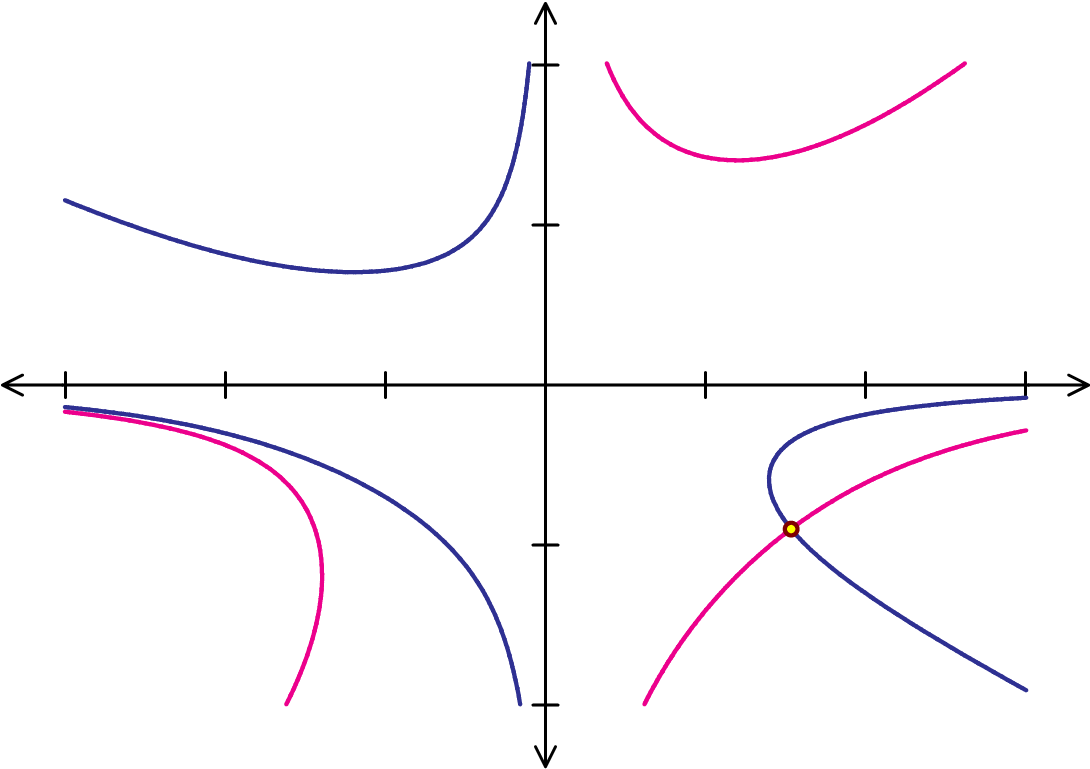}}
    \put(54,54){$-1$}  \put(118,54){$1$} 
    \put(97,35.5){$-1$}  \put(97,90){$1$} 
    \put(13,101){\Blue{$f$}}    \put(79,12){\Blue{$f$}}     \put(178,12){\Blue{$f$}}
    \put(42,12){\Magenta{$g$}} \put(116,12){\Magenta{$g$}} \put(160,112){\Magenta{$g$}} 
    \put(168,41.1){\vector(-1,0){29}} \put(170,38){$(1.53,-0.907)$}
    \put(97,126){$y$}  \put(180,70){$x$}
   \end{picture}
 \caption{Curves of polynomial system~\eqref{Eq3:ex_sparse}.}
 \label{F:sparse_curves}
 \end{figure}
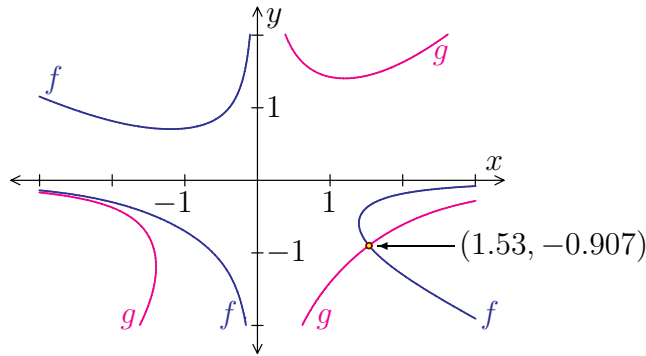
 The exponent vectors $\calA$ are the columns of the matrix 
 $(\begin{smallmatrix}2&1&0&1\\1&2&0&1\end{smallmatrix})$.
 The map $\varphi_\calA$ is 
\[
  (x,y)\ \longmapsto\ [x^2y\,:\ xy^2\,:\ 1\,:\ xy]\ \in\P^\calA\simeq \P^3\,.
\]
 Its image consists of those points $[z_{\MySmVec{2}{1}}:z_{\MySmVec{1}{2}}:z_{\MySmVec{0}{0}}:z_{\MySmVec{1}{1}}]$ with 
 $z_{\MySmVec{2}{1}}z_{\MySmVec{1}{2}}z_{\MySmVec{0}{0}}=z_{\MySmVec{1}{1}}^3\neq 0$, which is part of a cubic surface.
 The polynomial system~\eqref{Eq3:ex_sparse} corresponds to the two linear forms
\[
  z_{\MySmVec{2}{1}}+2z_{\MySmVec{1}{2}}-z_{\MySmVec{0}{0}}+z_{\MySmVec{1}{1}}\ =\ 
  z_{\MySmVec{2}{1}}-z_{\MySmVec{1}{2}}+2z_{\MySmVec{0}{0}}-z_{\MySmVec{1}{1}}\ =\ 0\,.
\]
 These define a line \defcolor{$\ell$} in $\P^3$. 
 Figure~\ref{F:cubic_surface} shows $\ell$ and (part of) the cubic surface.
 This is in the affine part of $\P^\calA$ where $z_{\MySmVec{1}{1}}\neq 0$ near the origin.
 The best view is from the $+-+$-orthant.
 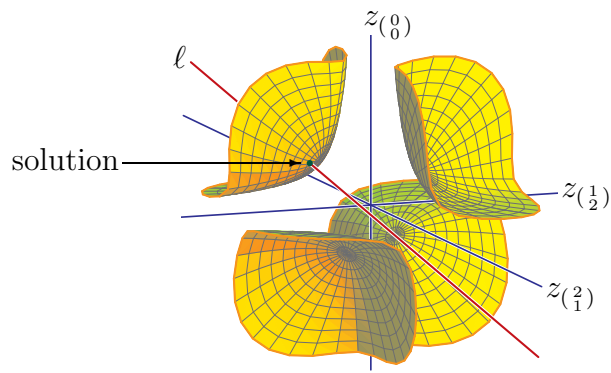
\begin{figure}[htb]
   \begin{picture}(213,135)(-63,0)
    \put(0,0){\includegraphics[width=5cm]{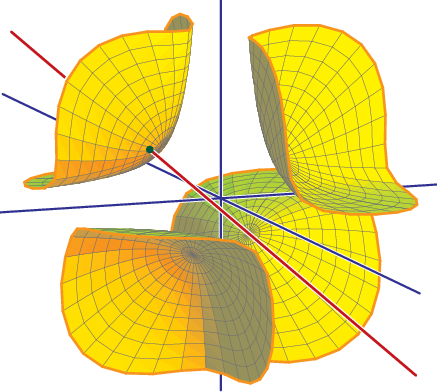}}
    \put(-3,115){$\ell$}
    \put(138,28){$z_{\MySmVec{2}{1}}$}  \put(144,66){$z_{\MySmVec{1}{2}}$} \put(69,131){$z_{\MySmVec{0}{0}}$}
    \put(-64,76.3){solution}\put(-22,78.7){\vector(1,0){67}}
   \end{picture}
 \caption{Linear section of cubic surface.}
 \label{F:cubic_surface}
 \end{figure}
 From this, we see that there is one real solution to the system~\eqref{Eq3:ex_sparse}.
 \qed
  \end{example}

Lemma~\ref{L3:linearSection} gives an interpretation for the number \defcolor{$d(\calA)$} of
solutions to a general system~\eqref{E3:PolySystem} with support $\calA$.
The \demph{degree, $\deg(X)$}, of a subvariety $X$ of $\P^m$ of dimension $n$ is the number of 
points in a linear section $L\cap X$ of $X$, where \defcolor{$L$} is a general linear subspace in $\P^\calA$ of
codimension $n$.  
Since a general linear subspace of codimension $n$ meets the toric variety $X_{\calA^+}$ only at points in the image
$\varphi_\calA((\C^*)^n)$, and by Exercise~\ref{Ex:Prinjective}, $\varphi_\calA$ is injective if and only if
$\Aff_\Z\calA=\Z^n$,  we deduce the following.

\begin{lemma}
 When the affine span of $\calA$ is $\Z^n$, $d(\calA)=\deg(X_{\calA^+})$.
\end{lemma}

We prove Kushnirenko's Theorem in the case when $\Aff_\Z\calA=\Z^n$ by showing that 
\[
  n!\cdot \Vol(\conv(\calA))\ =\  \deg(X_{\calA^+}) \,
\]
This proof is due to Khovanskii~\cite{Kh95} and the presentation is adapted from Chapter 3 of~\cite{IHP}, where the
general case of $\Aff_\Z\calA\subsetneq \Z^n$ is deduced from the case when $\Aff_\Z\calA=\Z^n$.

The \demph{homogeneous coordinate ring $\C[X]$} of a projective variety
$X\subset\P^\calA$  
is the quotient of the homogeneous coordinate ring $\C[z_a\mid a\in\calA]$ of $\P^\calA$
by the ideal $I_X$ of homogeneous polynomials vanishing on $X$.
These rings and ideals are graded by the total degree of the polynomials.
Writing \defcolor{$\C_d[X]$} for the $d$th graded piece of $\C[X]$, the \demph{Hilbert function $\HF_X(d)$} is the
function $d\mapsto\dim_\C \C_d[X]$. 

Hilbert proved that the Hilbert function for $d\gg 0$ is equal to a polynomial, which is now called 
the \demph{Hilbert polynomial $\HP_X(d)$} of $X$.
This encodes many numerical invariants of $X$.
For example, the degree of the Hilbert polynomial is the dimension $n$ of $X$
and its leading coefficient is $\frac{1}{n!}\deg(X)$.
For a discussion of Hilbert polynomials, see Section 9.3 of~\cite{CLO}.

We determine the Hilbert polynomial of the toric variety $X_{\calA^+}$.
Its homogeneous coordinate ring is the coordinate ring of $X_{\calA^+}\subset\C^\calA$.
By Corollary~\ref{Co:CoordinateRing}, this is $\C[\N\calA^+]$.
As the first coordinate $1$ of points of $\calA^+$ corresponds to the homogenizing parameter $t$ in~\eqref{Eq:liftMap},        
$\C[\N\calA^+]$ is graded by the first component of elements of $\N\calA^+$.
Thus $\C_d[\N\calA^+]$ has a basis $\{ (d,a) \in \N\calA^+\}$.
This is $\N\calA^+\cap d\conv(\calA^+)$, which is equal to $\defcolor{d\calA^+}$, the
set of $d$-fold sums of vectors in $\calA^+$.

\begin{example}\label{Ex:cuspidal_monoid}
 Consider this for the projectivization $X_{\calA^+}$ of the cuspidal cubic of Example~\ref{Ex:AffineTV}.
 Here, $\calA=\{0,2,3\}$ and $\varphi_\calA(s)=[1,s^2,s^3]\in\P^\calA$.
 Figure~\ref{F:cuspidal_monoid} shows its lift $\calA^+=(\begin{smallmatrix}1&1&1\\0&2&3\end{smallmatrix})$ and the
 submonoid $\N\calA^+$. 
 \begin{figure}[htb]
  \begin{picture}(156,70)(-43,0)
    \put(0,0){\includegraphics{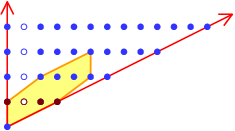}}
    \put(43,60){$\N\calA^+$}
    \put(-43,11){$\calA^+$}\put(-25,14){\vector(1,0){24}}
   \end{picture}
 \caption{Monoid generated by the lift of $\{0,2,3\}$.}
 \label{F:cuspidal_monoid}
 \end{figure}
 The open circles are points that do not lie in $\N\calA^+$. 
 The Hilbert function of $X_{\calA^+}$ has values $(1,3,6,9,12,\dotsc)$, so its Hilbert polynomial is $3d$.
 \qed
\end{example}

Projecting the set $d\calA^+$ to the last $n$ coordinates is a  
bijection with the set \defcolor{$d\calA$} of $d$-fold sums of vectors in $\calA$. 
These arguments show that 
\[
   \HF_{X_\calA}(d)\ =\ |d \calA|\,.
\]
Thus an upper bound on $\HF_{X_\calA}(d)$ is given by $|d\conv(\calA)\cap\Z^n|$, as $d\calA\subset d\conv(\calA)\cap\Z^n$.
Ehrhart~\cite{Eh62} (see also~\cite{BeRo07}) showed that for an integer polytope $P$, the counting function
\[
    \defcolor{E_P}\ \colon\ \N\ni d\ \longmapsto \ |dP\cap \Z^n|
\]
for the integer points contained in positive integer multiples of 
$P$ is a polynomial in $d$, now called the \demph{Ehrhart polynomial} of $P$.
The degree of $E_P$ is the dimension of the affine span of $P$.
When $P$ has dimension $n$, its leading coefficient is the volume of $P$.
For example, the Ehrhart polynomial of the interval $[0,3]=\conv\{0,2,3\}$ of length 3 is $3d+1$.

Now suppose that $P=\conv(\calA)$, the convex hull of $\calA$.
Since $d\calA\subset d\conv(\calA)\cap \Z^n$, we have 
the upper bound for $\HF_{X_\calA}(d)$,
 \begin{equation}\label{Eq3:Hilbert_compare}
   \HF_{X_\calA}(d)\ \leq\ E_{\conv(\calA)}(d)\,.
 \end{equation}
Note that we have this inequality for the cubic of Example~\ref{Ex:cuspidal_monoid}.

A lower bound for $\HF_{X_\calA}(d)$ is best expressed in terms of an inclusion.
Let $\defcolor{S_\calA}:=\R_\geq\calA^+\cap\Z^{1+n}$ be the monoid of all integer points that are in the nonnegative span of
$\calA^+$. 
The inequality~\eqref{Eq3:Hilbert_compare} arises from the inclusion $\N\calA^+\subset S_\calA$
by considering points with first
coordinate $d$.
We will produce a vector $v\in\N\calA^+$ and show that $v+S_\calA\subset\N\calA^+$, which we will use to show our lower bound.

Let $\defcolor{\calB}\subset S_\calA$ be the set of points $b\in\Z^n$ which may be written as
\[
   b\ =\ \sum_{a\in\calA} \lambda_a (1,a)\,,
\]
where $\lambda_a$ is a rational number in $[0,1)$.
For the set $\calA=\{0,2,3\}$ of Example~\ref{Ex:cuspidal_monoid}, $\calB$ is origin, together with the four points in
the interior of the hexagonal shaded region (a zonotope).
These are the columns of the matrix
$\MyVect{0&1&1&2&2}{0&1&2&3&4}$.

For each $b\in\calB$, fix an expression 
 \begin{equation}\label{E3:fixedExpression}
   b\ =\ \sum_{a\in\calA} \beta_a(b)(1,a) 
   \qquad (\beta_a(b)\in\Z)
 \end{equation}
as an integer linear combination of elements of $\calA^+$.
Let $-\nu$ with $\nu\geq 0$ be an integer lower bound for the coefficients $\beta_a(b)$ in these expressions for
the finitely many elements $b\in\calB$. 
For the set $\calA=\{0,2,3\}$, we may take these expressions to be 
$\MyVect{1}{1}=\MyVect{1}{0}-\MyVect{1}{2}+\MyVect{1}{3}$, 
$\MyVect{1}{2}=\MyVect{1}{2}$, 
$\MyVect{2}{3}=\MyVect{1}{0}+\MyVect{1}{3}$, 
and 
$\MyVect{2}{4}=2\MyVect{1}{2}$, 
so that $\nu=1$.
Finally, define 
\[
   \defcolor{v}\ :=\ \nu\cdot \sum_{a\in A}(1,a)\,.
\]
Its first coordinate is $\nu|\calA|$.
For the set $\calA=\{0,2,3\}$, this vector is 
$\MyVect{3}{5}=\MyVect{1}{0}+\MyVect{1}{2}+\MyVect{1}{3}$.

We claim that we have the inclusion of sets
 \begin{equation}
   \label{Eq:inclusion}
    v+S_\calA\ \subset\ \N\calA^+\ \subset\ S_\calA\,.
 \end{equation}
Comparing these sets at any level $d\geq \nu|\calA|$ gives the inequality
\[
  E_{\conv(\calA)}(d-\nu|\calA|) \ \leq\  \HF_{X_\calA}(d)\ \leq\ E_{\conv(\calA)}(d)\,,
\]
Since both the lower bound and the upper bound are polynomials in $d$ of the same degree and leading term, we deduce that
the Hilbert polynomial $\HP_{X_\calA}(d)$ has the same degree and leading term as the Ehrhart polynomial
$E_{\conv(\calA)}$.

Thus the Hilbert polynomial has degree $n$ and its leading coefficient is
the volume of $\conv(\calA)$.
Since the degree of $X_{\calA^+}$ is $n!$ times this leading coefficient, 
we conclude that the degree of $X_{\calA^+}$ is
\[
   n!\;\Vol(\conv(\calA))\,,
\]
which proves Kushnirenko's Theorem when $\Aff_\Z\calA=\Z^n$, given the inclusions~\eqref{Eq:inclusion}. 

We establish the first inclusion in~\eqref{Eq:inclusion}.
(The second was already discussed.)
Let $u\in v+S_\calA$.
Then $u-v\in S_\calA$ and so it has an expression 
\[
   u-v\ =\ \sum_{a\in\calA} \alpha_a (1,a) 
   \quad\mbox{ with }\quad \alpha_a\in\Q_\geq\,.
\]
Writing each coefficient $\alpha_a$ in terms of its fractional and integral parts gives
$\alpha_a=\lambda_a+\gamma_a$ where $\lambda_a\in[0,1)\cap\Q$ and $\gamma_a\in\N$.
Then
 \[
     u-v \ =\  \sum_{a\in\calA} \lambda_a (1,a)\ +\ \sum_{a\in\calA} \gamma_a(1,a)
        \ =\ b \ +\ c\,,
 \]
where $b\in\calB$ and $c\in\N\calA^+$.
Using the fixed expression~\eqref{E3:fixedExpression} for $b$, we have
 \[
   w\ =\  v + \sum_a \beta_a(b)(1,a)\ +\ c
    \ =\  \sum_s (\beta_a(b)+\nu)(1,a)\ +\ c\,,
 \]
which lies in $\N\calA^+$ as $-\nu\leq\beta_a$.
This establishes the inclusion of sets~\eqref{Eq:inclusion}
and completes the proof of Kushnirenko's Theorem when $\Aff_\Z\calA=\Z^n$.\qed\medskip

The vector $v$ used to establish the inclusion~\eqref{Eq:inclusion} may be replaced by a more economical vector.
For each $a\in A$, set $\defcolor{\nu_a}:=\max\{0,-\beta_a(b)\mid b\in\calB\}$.
If we set $\defcolor{v'}:=\sum_{a\in\calA} \nu_a(1,a)$, then the same argument shows that we still have an inclusion
$v'+S_\calA\subset\N\calA^+$. 
For $\calA=\{0,2,3\}$ with  $\calA^+=(\begin{smallmatrix}1&1&1\\0&2&3\end{smallmatrix})$ the new vector $v'$ is
$\MyVect{1}{2}$. 
Figure~\ref{F:cuspidal_cone}
 \begin{figure}[htb]
  \begin{picture}(204,92)(-43,0)
    \put(0,0){\includegraphics{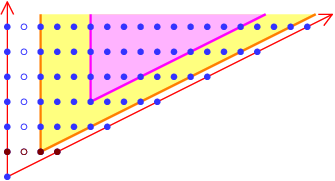}}
    \put(68,84){$\N\calA^+$}
    \put(-43,11){$\calA^+$}\put(-25,14){\vector(1,0){24}}
   \end{picture}
 \caption{Inclusions of cones for  $\calA=\{0,2,3\}$.}
 \label{F:cuspidal_cone}
 \end{figure}
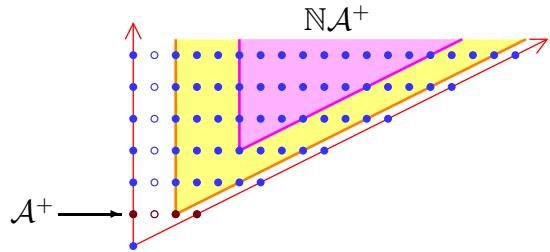
shows the monoid $S_\calA$ (all the circles, filled and unfilled), 
the monoid $\N\calA^+$ (the filled circles), 
the translate $\MyVect{1}{2}+S_\calA$  (larger shaded region), 
and finally the translate $\MyVect{3}{5}+S_\calA$  (smaller shaded region).
Observe that $\MyVect{1}{2}$ is the shortest vector such that the translate $\MyVect{1}{2}+S_\calA$ lies in $\N\calA^+$.

The expression $n!\;\Vol(\conv(\calA))$ in Kushnirenko's Theorem is often called the \demph{normalized volume} of
$\conv(\calA)$. 

\subsection*{Exercises}
\begin{enumerate}[1.] 

 \item \label{Exer:TwCubic} 
   Verify that the subgroup $\Z\calA$ for the set $\calA=(\begin{smallmatrix}3&2&1&0\\0&1&2&3\end{smallmatrix})$
   of Example~\ref{Ex:TwCubic} is a full rank (rank 2) subgroup of index 3 in $\Z^2$.
   You may find the map $\Z^2\to\Z$ given by $(p,q)\mapsto p+q$ useful; consider its kernel, image, and cokernel,
   and the restrictions to $\Z\calA$.

 \item Let $\calA\subset\Z^n$ be a finite set of points.
       Show that its lift $\calA^+\subset\Z^{1+n}$ spans $\Z^{1+n}$ if and only if the set of differences
       $\{ a-b\mid \forall a,b\in\calA\}$ spans $\Z^n$.

 \item \label{Exer:Lift}
   Prove that if a finite set $\calA\subset\Z^{n}$ lies on an affine hyperplane and $\rank(\Z\calA)=1+m$,
   then there is a basis for $\Z\calA$ identifying it with $\Z^{1+m}$ and a subset $\calB\subset\Z^m$ such that 
   $\calA=\calB^+$.

 \item \label{Exer:Row-span}
   Suppose that $\calA\subset\Z^{n}$ is represented by an integer matrix, $A$.
   Show that $\calA$ lies on an affine hyperplane if and only if the row space of $A$ in
   $\R^\calA$ has a vector with every coordinate 1. 

 \item
   Prove the equivalence of the two definitions,~\eqref{Eq:Affine-span} and~\eqref{Eq:Affine-coset}, of affine span.

 \item \label{Ex:Prinjective}
   Let $\calA\subset\Z^n$ be finite.
   Show that the map $\varphi_\calA\colon(\C^*)^n\to\P^\calA$ is injective if and only if the integral affine span $\Z\calA$ of
   $\calA$ is $\Z^n$.

 \item  \label{Exer:lift_hex}
  Give a spanning set of degree two generators for $I_\calA$, where $\calA$ is the lifted hexagon of Figure~\ref{F:lift}.
  Interpret each generator as a point common to the convex hulls of two disjoint subsets of $\calA$.

 \item
  Repeat Exercise~\ref{Exer:lift_hex} for (a) the lift of the cube and (b) the lift of the octahedron.
\[
   \includegraphics[height=60pt]{figures/Cube.eps}\qquad\qquad
   \raisebox{-10pt}{\includegraphics[height=80pt]{figures/Octahedron.eps}}
\]

 \item \label{Exer:HToricIdeals} 
   Do the homogeneous version of Exercise~\ref{Exer:ToricIdeals} from Section 1.
   For a polygon $P$, let $\calA_P:=P\cap\Z^2$ be the set of integer points in $P$, and $\calA^+_P$ the lift of these
   points to $\Z^3$.
   For each polygon $P$ below, identify homogeneous binomials that generate the homogeneous toric ideal $I_{\calA^+_P}$.
   For each generator, give the coincident convex combination of Proposition~\ref{P:coincidentConvex}, and the point of $P$
   to which it corresponds.
\[
  \includegraphics{figures/LatticeV.eps}\qquad\ 
  \includegraphics{figures/LatticeDiamond.eps}\qquad\ 
  \includegraphics{figures/LatticeP.eps}\qquad\ 
  \includegraphics{figures/LatticeHexagon.eps}\qquad\ 
  \includegraphics{figures/LatticeQ.eps}
\]
 Are homogeneous toric ideals always generated by quadratic binomials?
 \item 
   Show that the Euclidean volume of the simplex $\conv\{0,e_1,\dotsc,e_n\}$ is $\frac{1}{n!}$, where $e_i$ is the
   standard coordinate unit vector in $\R^n$.
   Harder: Prove that this is the minimum volume of any lattice simplex, and that all others have volume an integer multiple
   of $\frac{1}{n!}$.

 \item  Determine the volume of the Newton polytope of the Laurent polynomial
\[
   1+ x+2y+3z -4xyz + 5x^2y + 7yz^2 + 11 x^2z^2 + 13 xy^3z + 17 y^3z^2-8x^2y^2z^2\,.
\]
   Hint: use a computer algebra system to determine the number of solutions to a general sparse system with this 
   support and apply Kushnirenko's Theorem.
   Challenge: Can you use this method to prove the volume is what you computed?

 \item  \label{Ex:weightedTermOrder}
   Let $f \in \C[x^{\pm}]$ and let $\calA := \supp(f)$. 
     For any $w \in \R^n$, we have a partial term order on $\C[x^{\pm}]$ given by
\[
       x^a\ \defcolor{\prec_w}\ x^b\qquad \text{if}\qquad w\cdot a\ <\ w\cdot b\,.
\]
      Show that $\supp(\ini_w f) = \calA_w$, where $\calA_w$ is defined in~\eqref{Eq:calA_w}.

 \item For a challenging exercise, provide a proof of Lemma~\ref{L:BoundaryToricVars}.

 \item For an even more challenging exercise, provide a proof of Lemma~\ref{L:Birch}.
\end{enumerate}

\section{Toric Varieties From Fans}\label{S:three}

Affine toric varieties $X_\calA$ are given by a finite collection $\calA$ of integer vectors.
The ideal of an affine toric variety is generated by binomials coming from elements of the integer kernel of a linear map
determined by the set $\calA$.
When the set $\calA$ lies on an affine hyperplane, the toric variety $X_\calA$ is homogeneous and gives a projective toric
variety with structure corresponding to the polytope $\conv(\calA)$.
We give an abstract (not embedded) construction of a toric variety obtained by gluing affine toric varieties, where the
affine varieties and the gluing are encoded in an object from geometric combinatorics called a rational fan.

\begin{example}\label{Ex:P1}
 The projective line $\P^1$ is the projective toric variety given by columns $\calA$ of the matrix
 $(\begin{smallmatrix}1&0\\0&1\end{smallmatrix})$.
 The corresponding integer polytope is the convex hull of the two points of $\calA$, which 
 in an appropriate coordinate system is just the interval $[0,1]$.

 In its homogeneous coordinates $[x,y]$, $\P^1$ has two standard affine patches
 $\defcolor{\C_0}:=[x,1]$ for $x\in \C$ and $\defcolor{\C_\infty}:=[1,y]$ for $y\in\C$.
 Their intersection $\C_0\cap\C_\infty$ may be identified with $\C^*$; it is the points of either patch where the parameter
 ($x$ for $\C_0$ and $y$ for $\C_\infty$) does  not vanish.
 The point $[x,1]\in\C_0^*$ is identified with the  point $[1,x^{-1}]\in\C_\infty^*$.
 Consequently, $\P^1$ is the union of two copies of $\C$, $\C_0\sqcup\C_\infty$, glued along this common set.

 We organize this using subalgebras of $\C[x,x^{-1}]$.
 First, $\C_0=\spec\C[x]$, which is identified with $\Hom_{m}(\N,\C)$, where a point $[x,1]\in\C_0$ is the
 monoid homomorphism \defcolor{$f_x$} that sends the generator $1\in\N$ to $x\in\C$ and $f_x(0)=11$ as it is a homomorphism of monoids.
 Similarly, $\C_\infty=\spec\C[x^{-1}]=\Hom_{m}(-\N,\C)$.
 Here, a point $[1,y]$ is the monoid homomorphism \defcolor{$g_y$} that sends the generator $-1$ to $y\in\C$.
 Also, $\C^*=\spec\C[x,x^{-1}]$, which is $\Hom_{m}(\Z,\C)$, where a point $z\in\C^*$ is the monoid
 homomorphism \defcolor{$h_z$} that sends $0\mapsto 1$ and $1\mapsto z$.
 The restriction of $h_z$ to $\N$ gives the map $f_z$ and its restriction to $-\N$ is the map $g_{z^{-1}}$.
 \qed
\end{example}

%
%

\subsection{Cones and Fans}

We develop more geometric combinatorics needed for the remaining material on toric varieties, in the context of objects in
$\R^n$. 
For additional reference, we recommend the books of Ewald~\cite{Ewald} and Ziegler~\cite{Ziegler}.

Let $\calA\subset\R^n$ be a finite set.
As explained in Section~\ref{S:two}, its convex hull 
\[
   \defcolor{\conv(\calA)}\ :=\ 
    \Bigl\{ \sum_{a\in\calA} a\lambda_a \mid  \sum_{a\in\calA} \lambda_a=1
        \mbox{ and } 0\leq \lambda_a \mbox{ for all }a\in\calA\Bigr\}
\]
is a \demph{polytope}, $P$.
This polytope is a closed and bounded set, so for $w\in\R^n$, the linear function on $\R^n$ given by 
$x\mapsto w\cdot x$ is bounded on $P$, and thus has a maximum value, $h_P(w)$, on $P$.
This function $w\mapsto h_P(w)$ is the support function of $P$.
The subset $\defcolor{P_w}:=\{x\in P\mid w\cdot x=h_P(w)\}$ of $P$ where this maximum is attained
is the \demph{face} of $P$ \demph{exposed} by $w$, and is again a polytope, typically of smaller dimension.
It is the convex hull of \defcolor{$\calA_w$}, which is the set of points $a\in\calA$ where $w\cdot a=h_P(w)$.
(These notions were treated in Subsection~\ref{SS:ProjectiveTV}.)
The \demph{dimension} of a polytope is the dimension of its affine span.
A face $F$ of $P$ of dimension zero is a point and it is called a \demph{vertex} of $P$.
A face of dimension one is a line segment, and it is called an \demph{edge}.
A \demph{facet} of $P$ is a face $F$ with codimension one, $\dim F = \dim P{-}1$.

\begin{example}\label{Ex:Pyramid}
 A useful construction of one polytope from another is a pyramid.
 Suppose that $P$ is a polytope of dimension $n{-}1$, which we assume lies on a hyperplane \defcolor{$H$} defined by
 $w\cdot x=b$ in $\R^n$ for some $0\neq w\in\R^n$ ($H=\{x\in\R^n\mid w\cdot x=b\}$).
 For any point $o\in\R^n\smallsetminus H$, the \demph{pyramid} with base $P$ and apex $o\in\R^n$ is the convex hull of the
 polytope $P$ and the point $o$. 
 This pyramid has height $h:=\frac{1}{\|w\|}|b-w\cdot o|=|w\cdot(o - x)|/\|w\|$ for any $x\in H$, and its volume is
 $\frac{1}{n}h\Vol_{n-1}(P)$. 
\[
   \raisebox{-30pt}{\begin{picture}(95,72)
     \put(0,0){\includegraphics{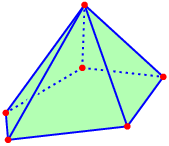}}
     \put(31,68){$o$}
     \put(83,20){$P$} \put(82,24){\vector(-1,0){20}}
   \end{picture}}
   \eqno{\MyDiamond}
\]
\end{example}

By the definition of the support function $h_P(w)$ of $P$, for any $w\in\R^n$, we have 
\[
   P\ \subset\ \{x\in\R^n \mid w\cdot x \leq h_P(w)\}\,.
\]
This set is a \demph{half space} and its boundary $\{x\mid w\cdot x= h_P(w)\}$ is a 
\demph{supporting hyperplane} of $P$.
Note that $P_w$ is the intersection of $P$ with the supporting hyperplane corresponding to $w$.
As a closed, convex body,  $P$ is the intersection of all half-spaces that contain it,
 \begin{equation}\label{Eq:dualDescription}
   P\ =\ \bigcap_{w\in\R^n}\{x\in\R^n \mid w\cdot x \leq h_P(w)\}\,.
 \end{equation}
For example, the lattice octahedron~\eqref{Eq:SomePolytopes} $\{(x,y,z)\in\R^3\mid |x|+|y|+|z|\leq 1\}$ is the intersection 
of the eight half spaces, 
$\{(x,y,z)\in\R^3\mid \pm x\pm y\pm z\leq 1\}$, one for each choice of the three signs $\pm$.
The lattice pentagon is the intersection of five half spaces,
\[
   \raisebox{-21pt}{\includegraphics{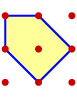}} \ =\ 
   \raisebox{-21pt}{\includegraphics{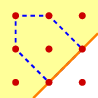}} \ \bigcap\ 
   \raisebox{-21pt}{\includegraphics{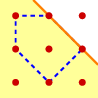}} \ \bigcap\ 
   \raisebox{-21pt}{\includegraphics{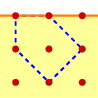}} \ \bigcap\ 
   \raisebox{-21pt}{\includegraphics{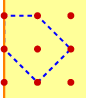}} \ \bigcap\ 
   \raisebox{-21pt}{\includegraphics{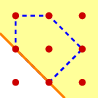}}\ .
\]
In these examples, the polytope $P$ is the intersection of finitely many half spaces, one for each facet of $P$.
This is true for all polytopes.

\begin{proposition}
 A polytope $P$ is the intersection of finitely many half spaces, one for each facet of $P$.
\end{proposition}

A \demph{polyhedron} is the intersection of finitely many half spaces, and a bounded polyhedron is a
polytope.
Here are four unbounded polyhedra in $\R$, $\R^2$, $\R^2$, and $\R^3$, respectively.
\[
  \raisebox{29pt}{\includegraphics{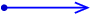}}\qquad
  \includegraphics{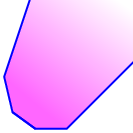}\qquad
  \includegraphics{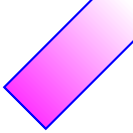}\qquad
  \includegraphics{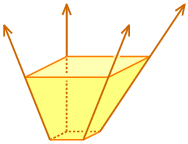}\qquad
\]
A polyhedron $P$ has a support function $h_P(w)$ that takes values in $\R\cup\{\infty\}$.
When $P$ is unbounded in the direction of $w$, then $h_P(w)=\infty$.
With this definition, the description~\eqref{Eq:dualDescription} holds for $P$.

A (convex) \demph{cone} is a polyhedron for which each supporting hyperplane contains the origin,
and is therefore a linear subspace.
Equivalently, a cone $\sigma$ is a polyhedron whose support function only takes values 0 and $\infty$.
The half spaces that define a cone all have the form
\[
   \{x\in\R^n \mid w\cdot x \leq 0\}\,.
\]
Such a half space forms an additive monoid under addition and its boundary hyperplane is a linear space consisting of the
invertible elements in this monoid.  

Consequently, a cone $\sigma$ is a monoid under addition and the intersection of its boundary hyperplanes is a linear
subspace $\ell$ of $\R^n$, called the \demph{lineality space} of $\sigma$.
The lineality space is the set of invertible elements in $\sigma$.
When the lineality space is the origin, the cone is \demph{pointed}
(also called \demph{strictly convex} or \demph{strongly convex}).
A face $\tau$ of a cone $\sigma$ is again a cone and 
the lineality space of $\sigma$ is its minimal face.
Figure~\ref{F:cones} shows four cones, two in $\R^2$ and two in $\R^3$.
\begin{figure}[htb]
  \begin{picture}(84,91)(0,-7)
    \put(0,0){\includegraphics{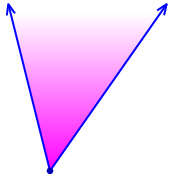}}
    \put(20,-9){$0$}
   \end{picture}
    \qquad
  \begin{picture}(85,91)(0,-7)
    \put(0,0){\includegraphics{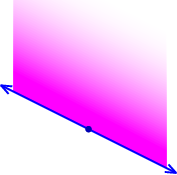}}
    \put(39,11){$0$}
    \put(1,25){$\ell$}
   \end{picture}
    \qquad
  \begin{picture}(85,91)(0,-7)
    \put(0,0){\includegraphics{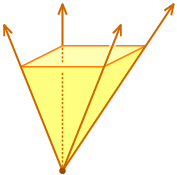}}
    \put(26,-9){$0$}
   \end{picture}
  \qquad
  \begin{picture}(104,91)(0,-7)
    \put(0,0){\includegraphics{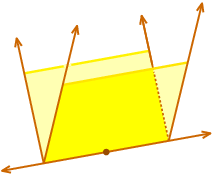}}
    \put(48,0){$0$}
    \put(2,7){$\ell$}
   \end{picture}
 \caption{Four cones.}
 \label{F:cones}
\end{figure}
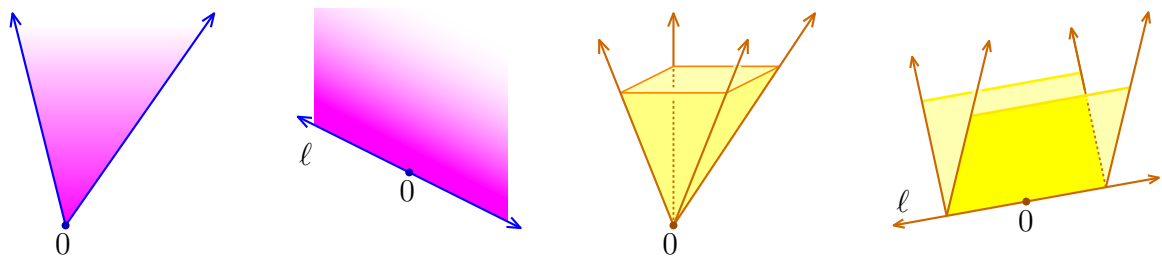
The first and the third are pointed, while the second and fourth have a one-dimensional lineality space.
The second is a half space.

The origin is the minimal face of a pointed cone $\sigma$. 
The \demph{rays} of a pointed cone $\sigma$ are its one-dimensional faces.
Each ray $\rho$ has the form $\R_{\geq} x$ for any nonzero element $x$ of $\rho$.
While a polytope is the convex hull of its vertices, a pointed cone is the sum of its rays.
For example, the third cone in Figure~\ref{F:cones} is
$\R_\geq \MySmTVec{0}{1}{2} +\R_\geq \MySmTVec{0}{-1}{2}+\R_\geq \MySmTVec{1}{0}{2} +\R_\geq \MySmTVec{-1}{0}{2}$.
More generally, any cone is a sum of rays.
The fourth cone in Figure~\ref{F:cones} is
$\R_\geq \MySmTVec{0}{1}{2} +\R_\geq \MySmTVec{0}{-1}{2}+\R_\geq \MySmTVec{1}{0}{0} +\R_\geq \MySmTVec{-1}{0}{0}$.

Another important object is a \demph{polyhedral complex}.
This is a collection \defcolor{$\calP$} of polyhedra in $\R^d$ such that every face of every polyhedron in $\calP$ is another
polyhedron in $\calP$ and the intersection of any two polyhedra $P,P'$ in $\calP$ is a
common face of each.
For example, of the four collections of vertices, line segments, and polyhedra below,
the first three are polyhedral complexes, while the last is not; the large
triangle does not meet either of the smaller triangles in one of its faces. 
\[
   \includegraphics{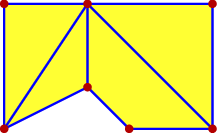}\qquad
   \includegraphics{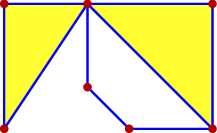}\qquad
   \includegraphics{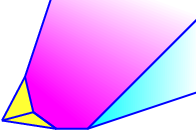}\qquad
   \includegraphics{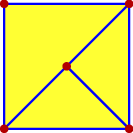}
\]
A polytope together with all of its faces forms a polyhedral complex.
The boundary of a polytope (all of its proper faces) forms a polyhedral complex.
For a less simple example, suppose that $o\in P$ is any point of a polytope $P$.
For every face $F$ of $P$ that does not contain $o$ we may consider the pyramid with base $F$ and apex $o$.
This collection of pyramids, their bases, and the apex $o$ forms a polyhedral subdivision of $P$.

The  \demph{support} of a polyhedral complex $\calP$ is the union of the polyhedra in $\calP$.
When the support of a polyhedral complex is a polyhedron $P$, the complex is a \demph{subdivision} of $P$.
When the support is a polytope $P$, we have the following formula for the volume of $P$,
\[
   \Vol_n(P)\ =\ \sum_{Q\in\calP} \Vol_n(Q)\,.
\]

When every polytope in a polyhedral complex $\calP$ is a
simplex, we say that $\calP$ is a \demph{triangulation} of its support.
Of the four polyhedral subdivisions below, the last two are triangulations.
\[
   \includegraphics{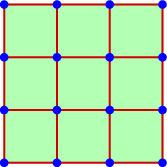}\qquad
   \includegraphics{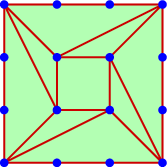}\qquad
   \includegraphics{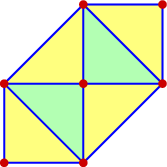}\qquad
   \includegraphics[height=80pt]{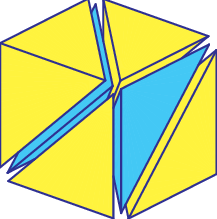}
\]

Our last object in this tour of geometric combinatorics is a \demph{fan}, which is 
a polyhedral complex, all of whose polyhedra are cones.
If the support of a fan is the ambient space, then the fan is said to be \demph{complete}.
Figure~\ref{F:fans} shows some fans in $\R^2$ and $\R^3$.
\begin{figure}[htb]
   \includegraphics{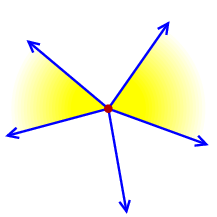}\qquad
   \includegraphics{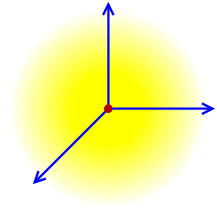}\qquad
   \includegraphics{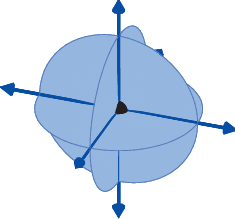}
 \caption{Three fans.}
 \label{F:fans}
\end{figure}
The second is complete, and the third is complete, if we include the eight implied open cones.

Given a polytope $P$ in $\R^n$, define an equivalence relation on the dual $\R^n$ by 
\defcolor{$v\sim w$} if and only if $P_v=P_w$, so that $v$ and $w$ expose the same face of $P$.
The closure of each equivalence class is a cone in $\R^n$, and these cones together form the 
\demph{(outer) normal fan} to the polytope $P$, which is a complete fan.
The rays of the normal fan expose facets of $P$.

\begin{example}\label{Ex:NormalFans}
The third fan in Figure~\ref{F:fans} is the outer normal fan to the lattice cube~\eqref{Eq:SomePolytopes}.
Indeed, the standard unit vectors $\MySmTVec{1}{0}{0}$, $\MySmTVec{0}{1}{0}$, and $\MySmTVec{0}{0}{1}$, together with their 
negatives, $\MySmTVec{-1}{0}{0}$, $\MySmTVec{0}{-1}{0}$, and $\MySmTVec{0}{0}{-1}$, expose the six facets of the cube.
An edge between two facets exposed by vectors $v$ and $w$ is exposed by any vector $\lambda v+\mu w$ where $\lambda,\mu>0$,
and all vectors in the interior of each orthant expose the same vertex.

For more examples, we display a regular heptagon and a lattice hexagon, together with their normal fans.
\[
  \includegraphics{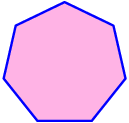}\qquad
  \raisebox{-10pt}{\includegraphics{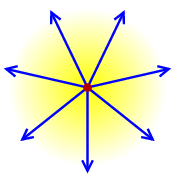}}\qquad\qquad
  \includegraphics{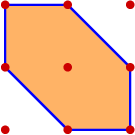}\qquad
  \raisebox{-10pt}{\includegraphics{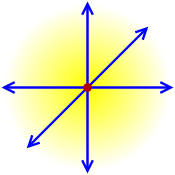}}
\]
Here are two views of the same polytope in $\R^3$, together with its normal fan having the same orientation.
\[
  \raisebox{10pt}{\includegraphics[height=70pt]{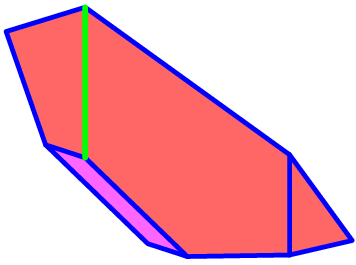}}\qquad
  \includegraphics[height=90pt]{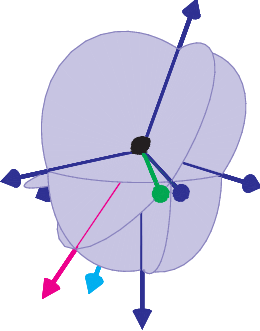}\qquad\qquad
  \raisebox{5pt}{\includegraphics[height=80pt]{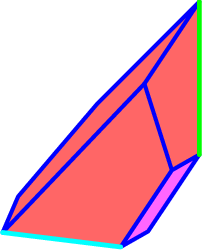}}\qquad
  \includegraphics[height=90pt]{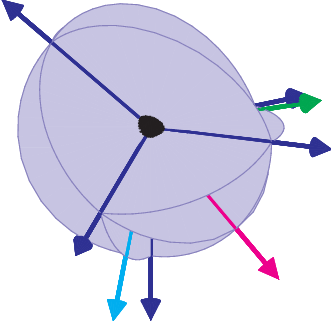}
\]
The \Magenta{magenta ray} in the normal fan is normal to the \Magenta{magenta facet}, 
the \Green{green ray} exposes the \Green{green edge}, and
the \Cyan{cyan ray} exposes the \Cyan{cyan edge}.
\qed
\end{example}




\subsection{Toric Varieties From Fans}

We give a construction of a toric variety by gluing affine toric varieties together
along common open subsets.
This is the original construction/definition of a toric variety from~\cite{Demazure}.
As in Sections~\ref{S:one} and~\ref{S:two}, we work with the torus $(\C^*)^n$, its group
$N=\Hom(\C^*,(\C^*)^n)\simeq\Z^n$ of cocharacters, and its group
$M=\Hom((\C^*)^n,\C^*)\simeq\Z^n$ of characters.
Let $\defcolor{N_\R}:=\R\otimes_\Z N\simeq\R^n$ be the real vector space spanned by the cocharacters
and $\defcolor{M_\R}:=\R\otimes_\Z M\simeq\R^n$ be the real vector space spanned by the characters.
Note that $N\subset N_\R$ and $M\subset M_\R$ in the same way as $\Z^n\subset\R^n$.
We write $\langle\bullet,\bullet\rangle\colon N_\R\times M_\R \to \R$ for the pairing
between $N_\R$ and $M_\R$, extending that between $N$ and $M$.

A (rational) fan $\Sigma\subset N_\R$ is a fan in $ N_\R$ in which every cone is defined by
inequalities coming from elements $a$ of $M$.
The half spaces defining the cones all have the form
\[
   \{ w\in  N_\R\mid \langle w,a\rangle\geq 0\}
   \qquad\mbox{ for some }\quad a\in M\,.
\]
The linear span of a cone $\sigma$ in a rational fan $\Sigma$ is a rational linear space in that it is spanned by its
intersection with $N$. 
The notions of rational fan and rational linear subspace also make sense in $ M_\R$.
We will assume that all cones in $\Sigma$ are pointed, as this simplifies the exposition.
In general, all cones in a rational fan have the same lineality space, $L_\R$, which contains a full rank sublattice
$L=N\cap L_\R$.
Replacing $N$ by $N/L$, $M$ by $L^\perp$, and every cone $\sigma$ of $\Sigma$ by its image in $N_\R/L_\R$, we obtain a
rational fan $\Sigma/L_\R$ that is pointed.
The price we pay for this is that the torus for the resulting toric varieties is identified with its dense orbit, which
leads to a loss of flexibility in our notion of toric variety:
E.g.\ the closure of a torus orbit in such a toric variety is a subvariety that is a toric variety---but for a different torus.

Given a rational cone $\sigma\subset N_\R$, its \demph{dual cone} is\
\[
   \defcolor{\sigma^\vee}\ :=\ \{x\in M_\R\,\mid\, \langle w,x \rangle \geq 0 \mbox{ for }w\in\sigma\}\,. 
\]
This rational cone has full dimension in $ M_\R$.
Its lineality space is \defcolor{$\sigma^\perp$}, the annihilator of
$\sigma$, which has dimension $n{-}\dim\sigma$.
This is the set of $x\in M_\R$ such that $\langle w,x\rangle=0$ for all $w\in\sigma$. 

\begin{example}
 The cone $\R_\geq\MySmVec{1}{0}+\R_\geq\MySmVec{0}{1}$ (the first quadrant in $\R^2$) has dual cone 
 $\R_\geq\MySmVec{1}{0}+\R_\geq\MySmVec{0}{1}$, the first quadrant in the dual $\R^2$.
 More interesting is the cone $\sigma=\R_\geq\MySmVec{1}{2}+\R_\geq\MySmVec{2}{1}$,
 whose dual cone is $\sigma^\vee=\R_\geq\MySmVec{2}{-1}+\R_\geq\MySmVec{-1}{2}$.
 We display all three cones below.
 \begin{equation}\label{Eq:conesAndDuals}
   \raisebox{-43pt}{\begin{picture}(64,97)
    \put(0,30){\includegraphics{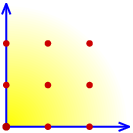}}
    \put(10,15){self-dual}
   \end{picture}
   \qquad\qquad
   \begin{picture}(64,97)
    \put(0,30){\includegraphics{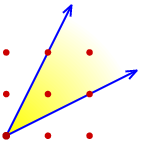}}
    \put(29,17){$\sigma$}
   \end{picture}
   \qquad\quad
   \begin{picture}(64,97)
    \put(0,0){\includegraphics{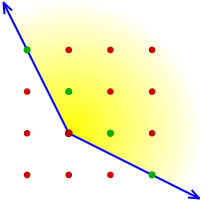}}
    \put(0,17){$\sigma^\vee$}
   \end{picture}}
 \end{equation}
 We give an example in $\R^3$.
 The two-dimensional cone $\sigma=\R_\geq\MySmTVec{1}{0}{0}+\R_\geq\MySmTVec{0}{1}{0}$ has dual cone
 $\sigma^\vee=\R_\geq\MySmTVec{1}{0}{0}+\R_\geq\MySmTVec{0}{1}{0}+\R\MySmTVec{0}{0}{1}$, whose
 lineality space is the vertical axis.
\[
   \raisebox{-40pt}{
    \begin{picture}(64,84)
     \put(0,14.5){\includegraphics{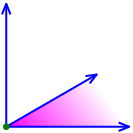}}
     \put(25,5){$\sigma$}
    \end{picture}
   \qquad\qquad
   \begin{picture}(64,84)
     \put(0,0){\includegraphics{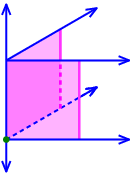}}
     \put(23,4){$\sigma^\vee$}
   \end{picture}}
   \eqno{\MyDiamond}
\]
\end{example}

Given a pointed rational cone $\sigma\subset N_\R$, set
\[
  \defcolor{S_\sigma}\ :=\   \sigma^\vee\cap M\ =\ 
   \{a\in M\,\mid\, \langle w,a \rangle \geq 0 \mbox{ for }w\in\sigma\}\,, 
\]
which is a submonoid of $M$.
Its group of invertible elements is the free abelian group
$\sigma^\perp\cap M$ which has rank $n-\dim\sigma$, the dimension of the lineality space $\sigma^\perp$ of $\sigma^\vee$.
Finally, let $\defcolor{V_\sigma}:=\Hom_{m}(S_\sigma,\C)$, the set of monoid homomorphisms from $S_\sigma$ to $\C$.
As we saw in Section~\ref{S:one}, this is equal to the spectrum of the monoid algebra $\C[S_\sigma]$ and so
$V_\sigma$ is an affine toric variety.
(This requires a theorem of Hilbert that such a monoid is finitely generated.)
However, it is not naturally embedded in an affine space $\C^\calA$.
For this, we need to choose a set $\calA\subset S_\sigma$ that generates $S_\sigma$ as a monoid, so that $S_\sigma=\N\calA$. 
This is this is equivalent to the surjectivity of the map $\N^\calA\rightarrow S_\sigma$ given by 
$(n_a\mid a\in\calA)\mapsto\sum_a an_a$.

\begin{example}\label{Ex:coneOverTwCubic}
 For the cone  $\sigma=\R_\geq\MySmVec{1}{2}+\R_\geq\MySmVec{2}{1}$, a generating set for $S_\sigma$ is 
 $\calA=\MySmVec{2&1&0&-1}{-1&0&1&2}$, which is highlighted in~\eqref{Eq:conesAndDuals}.
 The map
\[
   \Hom_{m}(S_\sigma,\C)\ \ni\ f\ \longmapsto\ 
    \bigl(f\MySmVec{2}{-1},f\MySmVec{1}{0},f\MySmVec{0}{1},f\MySmVec{-1}{2}\bigr)\ \in\ \C^\calA
\]
 is an embedding of $V_\sigma$ into $\C^\calA$.
 The image satisfies the equations
\[
   z_{\MySmVec{2}{-1}}z_{\MySmVec{0}{1}} - z_{\MySmVec{1}{0}}^2\,,\ \ 
   z_{\MySmVec{2}{-1}}z_{\MySmVec{-1}{2}} - z_{\MySmVec{1}{0}}z_{\MySmVec{0}{1}}\,,\ \mbox{ and }\ 
   z_{\MySmVec{1}{0}}z_{\MySmVec{-1}{2}} - z_{\MySmVec{0}{1}}^2\,,
\]
 which are given by the relations in $S_\sigma$ satisfied by elements of $\calA$.
 These are the same equations as~\eqref{Eq:TwistedCubicEquations},
 and so this affine toric variety is the affine cone over the rational normal curve.
\qed
\end{example}

Suppose that $\tau$ is a face of a rational cone $\sigma\subset N_\R$.
Then $\sigma^\vee\subset\tau^\vee$, as $\tau^\vee$ involves fewer inequalities.
This gives an inclusion $S_\sigma\subset S_\tau$ of monoids, and that in turn gives an inclusion of affine toric
varieties associated to these cones, 
\[
   V_\tau\ =\ \Hom_{m}(S_\tau,\C)\ \subset\ 
   \Hom_{m}(S_\sigma,\C)\ =\ V_\sigma\,.
\]
Here, the inclusion is given by restricting a monoid homomorphism $g\colon S_\tau\to\C$ to the submonoid $S_\sigma$.

\begin{definition}
 Let $\Sigma\subset N_\R$ be a rational fan.
 The toric variety \demph{$X_\Sigma$} associated to $\Sigma$ is obtained from the collection
 $\{V_\sigma\mid\sigma\in\Sigma\}$ of affine toric varieties associated to cones in the fan $\Sigma$ by gluing along the
 natural inclusions $V_\tau\subset V_\sigma$ whenever $\tau,\sigma$ are cones in $\Sigma$ with $\tau$ a face of
 $\sigma$. 
 As the minimal cone in $\Sigma$ is the origin $\bzero$ and $\bzero^\perp=M_\R$ so that $S_{\bzero}=M$.
 Thus $V_\bzero$ is the torus $\Hom_{m}(M,\C)=\Hom_g(M,\C^*)$.
 This lies in every affine toric patch $V_\sigma$ and the gluing is torus-equivariant.
 This shows that the toric variety $X_\sigma$ has an action of this torus.\qed
\end{definition}

\begin{example}
 Example~\ref{Ex:P1} constructs the projective line $\P^1$ from the fan $\Sigma\subset\R$
 whose cones are $\sigma=\R_\geq$, the nonnegative real numbers, 
 $\rho=\R_\leq$, the nonpositive real numbers, and $\tau=\{0\}$, the origin.
\[
   \begin{picture}(124,23)(0,-9)
    \put(30,9){$\rho$}    \put(59,9){$\tau$}    \put(90,9){$\sigma$}
   
    \put(0,0){\includegraphics{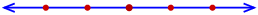}}
                           \put(59,-9){$0$}
   \end{picture}
\]
 Then $V_\sigma=\Hom_{m}(\N,\C)$, $V_\rho=\Hom_{m}(-\N,\C)$, and 
 $V_\tau=\Hom_{m}(\Z,\C)$, which are the sets $\C_0$, $\C_\infty$, and $\C^*$ of Example~\ref{Ex:P1}, and the gluing is
 the same as in Example~\ref{Ex:P1}.\qed
\end{example}

We give a detailed construction of a less trivial toric variety $X_\Sigma$ in Section~\ref{SS:DoublePillow}.\smallskip

We connect abstract toric varieties to the projective toric varieties of Section~\ref{S:two}.
For this, let $P\subset M_\R$ be a polytope with vertices in $M$ and set $\calA:=P\cap M$.
Write \defcolor{$X_P$} for the projective toric variety $X_{\calA^+}$.
Let $\defcolor{\Sigma_P}\subset N_\R$ be the outer normal fan of $P$.
Each cone $\sigma\in\Sigma_P$ corresponds to a unique face \defcolor{$P_\sigma$} of $P$.
(Recall that each cone $\sigma\in\Sigma_P$ is the closure of the set of points of $N_\R$ that expose a given face of $P$.) 
More specifically, let \defcolor{$\sigma^\circ$} be the \demph{relative interior} of $\sigma$, the set-theoretic
difference of $\sigma$ with the union of its proper faces.
Then $P_\sigma$ is the face of $P$ exposed by any point $w\in\sigma^\circ$.
Also, the linear span of differences $b-a$ of points $a,b\in P_\sigma$ is the lineality space of $\sigma^\vee$.

For a cone $\sigma\in\Sigma_P$, we define a map $\phi_\sigma\colon V_\sigma\to\P^\calA$ whose image lies in the
projective toric variety $X_P$.
Choose any point $b\in P_\sigma\cap M$.
By the definition of $P_\sigma$, if $w\in\sigma$ and $a\in\calA(=P\cap M)$, then 
$\langle w, b\rangle\geq\langle w, a\rangle$, so that $\langle w, b-a\rangle\geq 0$.
Since this inequality holds for all $w\in\sigma$, we have that $b-a\in \sigma^\vee\cap M = S_\sigma$.
For an element $f\in V_\sigma = \Hom_{m}(S_\sigma,\C)$, define $\defcolor{\varphi_{b,\sigma}(f)}\in\P^\calA$ to be the
point $[ f(b-a)\mid a\in\calA]$.
We note that the sign here, $b-a$ for $a\in\calA$, is to conform with our use of the outer normal fan.
For the inner normal fan, use $a-b$ for $a\in\calA$.

\begin{lemma}\label{L:FirstCheck}
 For any two elements $b,b'\in P_\sigma\cap M$ and $f\in V_\sigma$, we have that 
 $\varphi_{b,\sigma}(f) = \varphi_{b',\sigma}(f)$ as points in $\P^\calA$.
 For any $f\in V_\sigma$, $\varphi_{b,\sigma}(f)\in X_P$.
 For any point $x \in X_P$ with $x_c\neq 0$ for $c\in P_\sigma\cap M$, there is a point $f\in V_\sigma$ with 
 $x=\varphi_{b,\sigma}(f)$.
\end{lemma}

\begin{proof}
 For $b,b'\in P_\sigma\cap M$, $b-b'\in\sigma^\perp$ so that if $f\in V_\sigma$, then $f(b-b')$ is a nonzero scalar.
 For any $a\in\calA$, we have 
\[
   f(b'-a)\ =\ f(b'-b + b-a)\ =\ f(b'-b)\cdot f(b-a)\,,
\]
 as $f\in\Hom_{m}(S_\sigma,\C)$.
 Thus as points of $\C^\calA$, $\varphi_{b',\sigma}(f) = f(b'-b) \varphi_{b,\sigma}(f)$, so they give the same point in
 $\P^\calA$. 

 To see that  $\varphi_{b,\sigma}(f)\in X_P=X_{\calA^+}$, we show that 
 $\varphi_{b,\sigma}(f)\in\calV(I_{\calA^+})$.
 By Theorem~\ref{Th:Lin_Span_Toric}, it suffices to check that each binomial $z^u-z^v$ in $I_{\calA^+}$ vanishes at
 $\varphi_{b,\sigma}(f)$. 
 Suppose that $u,v\in\N^\calA$ satisfies $\calA^+u=\calA^+v$.
 Then $\calA u=\calA v$ and $|u|=|v|$.
 We compute
 \begin{multline*}
   \qquad
   \bigl(\varphi_{b,\sigma}(f) \bigr)^u\ =\ 
   f\Bigl(\sum_{a\in\calA} u_a(b-a)\Bigr)\ =\ f(|u|b - \calA u)\\
    =\ 
   f(|v|b - \calA v)\ =\ f\Bigl(\sum_{a\in\calA} v_a(b-a)\Bigr)\ =\ 
   \bigl(\varphi_{b,\sigma}(f) \bigr)^v\,, 
  \qquad
 \end{multline*}
 and so  $z^u-z^v$ vanishes at $\varphi_{b,\sigma}(f)$.
 Thus $\varphi_{b,\sigma}(f)\in\calV(I_{\calA^+})$.
 But this equals $X_{\calA^+}$ as $I_{\calA^+}$ is prime and hence radical, by Theorem~\ref{Th:Lin_Span_Toric}.

 Finally, suppose that $x=[x_a\mid a\in\calA]\in X_{\calA^+}$ is a point with $x_c\neq 0$ for some $c\in P_\sigma\cap M$.
 Choose any $b\in P_\sigma\cap M$ and define the monoid homomorphism $f\colon \N\{b-a\mid a\in\calA\}\to \C$  by
 $f(b-a)=x_a x_b^{-1}$, and extend linearly. 
 This is well-defined as $x\in X_{\calA^+}$ and so $x^u=x^v$ and whenever $\calA^+ u=\calA^+v$.
 If $S_\sigma=\N\{b-a\mid a\in\calA\}$, this completes the proof.
 Otherwise, as $S_\sigma=M\cap\Q\{b-a\mid a\in\calA\}$ and $\C$ is algebraically closed, 
 we may extend $f$ to a monoid homomorphism of $S_\sigma$.
\end{proof}

Thus we obtain a well-defined map $\defcolor{\varphi_\sigma}\colon V_\sigma\to X_P\subset\P^\calA$ whose image is the set of points
of $X_P$ whose coordinates indexed by elements of $P_\sigma\cap M$ are nonzero.

\begin{lemma}\label{L:SecondCheck}
 Let $\sigma,\tau\in\Sigma$ be cones with $\tau\subset\sigma$.
 For any element $f\in V_\tau$, we have $\varphi_\tau(f)=\varphi_\sigma(f)$, where we apply $\varphi_\sigma$ to the
 image of $f$ under the natural inclusion $V_\tau\subset V_\sigma$.
\end{lemma}

\begin{proof}
 This is tautological.
 If $f\in V_\tau$, then its image in $V_\sigma$ is obtained by restricting $f$ from $S_\tau$ to the points of $S_\sigma$.
 Choosing $b\in P_\tau\cap M\subset P_\sigma\cap M$, we have that $\{b-a\mid a\in\calA\}\subset S_\sigma\subset S_\tau$,
 so that $\varphi_{b,\tau}(f)=\varphi_{b,\sigma}(f)$, which completes the proof.
\end{proof}

By Lemma~\ref{L:SecondCheck}, the map from the disjoint union of the $V_\sigma$ for $\sigma\in\Sigma$ to $X_P$ given by the
collection of maps  
$\{\varphi_\sigma\mid \sigma\in\Sigma\}$ agrees on the inclusions $V_\tau\subset V_\sigma$ given by inclusions
$\tau\subset\sigma$ of cones in $\Sigma$.
Thus it induces a (surjective) map $\varphi_P\colon X_\Sigma\to X_P$.
This map is in fact an isomorphism of algebraic varieties.

\begin{remark}
 A map similar to the  map $\varphi_P\colon X_\Sigma\to X_P$ may also be defined for any set $\calA\subset P\cap M$ such that 
 $\conv(\calA)=P$.
 Its image will be the projective toric variety $X_{\calA^+}$.
\qed
\end{remark}

\subsection{The Double Pillow}\label{SS:DoublePillow}

We illustrate the construction of toric varieties for the normal fan $\Sigma\subset\R^2$ of the diamond, 
\raisebox{-2pt}{\MyDiamond},
which is the convex hull of the column vectors of the matrix $\MyVect{0&-1&0&1&0}{-1&0&0&0&1}$.
We display this lattice polygon and its normal fan $\Sigma$.
 \begin{equation}\label{Eq:FanDiamond}
  \raisebox{-37pt}{\includegraphics{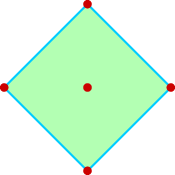}}
   \qquad\qquad
  \raisebox{-51pt}{\begin{picture}(132,112)(5,0)
    \put(5,0){\includegraphics{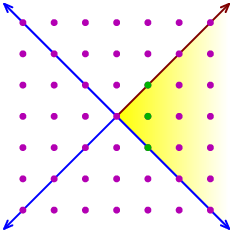}}
    \put(71,109){$\tau$}\put(77,109){\vector(1,-1){17}}
    \put(132,62){$\sigma$}\put(130,64){\vector(-1,0){30}}
  \end{picture}}
 \end{equation}
The fan $\Sigma$ has four rays $\R_\geq\MyVect{\pm1}{\pm1}$ one for each of four choices of signs, and four
two-dimensional cones spanned by adjacent rays.

Each two-dimensional cone $\sigma$ is self-dual and all are isomorphic.
Thus $X_\Sigma$ is obtained by gluing together four isomorphic affine toric varieties $V_\sigma$, as
$\sigma$ ranges over the two-dimensional cones in $\Sigma$.
A complete picture of the gluing involves the affine varieties $V_\tau$, where $\tau$ a ray of $\Sigma$.
We next describe these two toric varieties $V_\sigma$ and
$V_\tau$, for $\sigma$ a two-dimensional cone of $\Sigma$ and $\tau$ a ray of
$\Sigma$. 

Let $\sigma$ be the shaded cone in~\eqref{Eq:FanDiamond}.
Since $\sigma=\sigma^\vee$, we see that $\sigma^\vee\cap\mathbb{Z}^2$ is
minimally generated by the column vectors $\calA$ of the matrix $\MyVect{1&1&1}{-1&1&0}$ highlighted
in~\eqref{Eq:FanDiamond} and so 
$V_\sigma$ is isomorphic to the affine toric variety $X_\calA$, which is the closure in $\C^3$ of the image of the map 
\[
  \varphi\ \colon\ (s,t)\ \longmapsto\ (st^{-1}, st, s)\,,
\]
and is defined by the equation $z_{\MySmVec{1}{-1}}z_{\MySmVec{1}{1}}=z_{\MySmVec{1}{0}}^2$.
This is a cone in $\C^3$.
We display its real points (a right circular cone) in $\R^3$ below at left.
\begin{equation}\label{Eq:PiecesOfPillow}
 \raisebox{-51pt}{ \begin{picture}(144,102)
  \put(0,0){\includegraphics[width=2in]{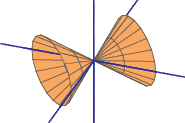}}
   \put(109,90){$z_{\MySmVec{1}{1}}$}
   \put( 75,92){$z_{\MySmVec{1}{0}}$}
   \put(138,42){$z_{\MySmVec{1}{-1}}$}
   \put(92,20){$V_\sigma$}
  \end{picture}
   \qquad\qquad
  \begin{picture}(144,102)
   \put(0,0){\includegraphics[width=2in]{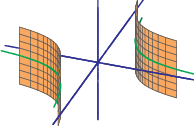}}
   \put(112,94){$z_{\MySmVec{-1}{1}}$}
   \put( 75,96){$z_{\MySmVec{1}{0}}$}
   \put(141,27){$z_{\MySmVec{1}{-1}}$}
   \put(101,18){$V_\tau$}
  \end{picture}}
\end{equation}

Let $\tau$ be the ray generated by $\MySmVec{1}{1}$, which is a face of $\sigma$.
Then $\tau^\vee$ is the half-space $\{(u,v)\in\R^2\mid u+v\geq 0\}$, 
which is the union of both  two-dimensional cones in $\Sigma$ containing $\tau$.
Since $\tau^\vee\cap\mathbb{Z}^2$ has generators the column vectors $\calB$ of the matrix $\MySmVec{1&-1&1}{-1&1&0}$, 
$V_\tau$ is isomorphic to the affine toric variety $X_\calB$, which is the closure in $\C^3$ of the image of the map 
\[
  \varphi\ \colon\ (s,t)\ \longmapsto\ (st^{-1}, s^{-1}t, s)\,,
\]
and has equation $z_{\MySmVec{1}{-1}}z_{\MySmVec{-1}{1}}=1$.
This is the cylinder with base the hyperbola $z_{\MySmVec{1}{-1}}z_{\MySmVec{-1}{1}}=1$ in the
$z_{\MySmVec{1}{-1}},z_{\MySmVec{-1}{1}}$-plane, which is shown in~\eqref{Eq:PiecesOfPillow} (in $\R^3$) at right.

We describe the gluing.
We have that $V_\tau\subset V_\sigma$ and they both
contain the torus $(\C^*)^2$.
In each, this common torus is its intersection with the complement of the coordinate planes in the given embedding, and
the boundary of the torus is its intersection with the coordinate planes.
The boundary of $(\C^*)^2$ in the cylinder $V_\tau$ is the curve $z_{\MySmVec{1}{0}}=0$ and
$z_{\MySmVec{1}{-1}}z_{\MySmVec{-1}{1}}=1$, which  
is displayed in~\eqref{Eq:PiecesOfPillow} on the picture of $V_\tau$.
Also, $t\neq 0$ on this cylinder.
The boundary of  $(\C^*)^2$ in the cone is the union of the $z_{\MySmVec{1}{-1}}$- and $z_{\MySmVec{1}{1}}$-axes.  
Since $t=z_{\MySmVec{1}{1}}/z_{\MySmVec{1}{0}}$ on the cone, the locus where $t=0$ is the $z_{\MySmVec{1}{1}}$-axis.
Thus $V_\tau$ is naturally identified with the complement of the $z_{\MySmVec{1}{1}}$-axis in $V_\sigma$ while the curve
$z_{\MySmVec{1}{0}}=0$ and $z_{\MySmVec{1}{-1}}z_{\MySmVec{-1}{1}}=1$ in 
$V_\tau$ is identified with the $z_{\MySmVec{1}{-1}}$-axis in $V_\sigma$.

If $\tau'$ is the other ray of $\sigma$, then $V_{\tau'}$ ($\simeq V_\tau$) 
is identified with the complement of the $z_{\MySmVec{1}{-1}}$-axis in $V_\sigma$. 
A convincing understanding of this gluing procedure may be obtained by considering
an image of the real points of the toric variety $X_{\MySmDiamond}$ in a projection to $\P^3$.
(Recall that $X_\Sigma$ has a map to the projective toric variety $X_{\MySmDiamond}\subset\P^\calA=\P^4$ which is an isomorphism.) 
The map $\P^\calA\to\P^3$ is given by the points
$[1:\pm1:0:0]$ and $[1:0:\pm1:0]$ associated to the vertices 
$(\pm1,0)$ and $(0,\pm1)$ of $\raisebox{-2pt}{\MyDiamond}$, and the vertical point at infinity $[0:0:0:1]$ associated to
its center. 
(Here, the plane at infinity is $[0:x:y:z]$.)
The image of the toric variety $X_{\MySmDiamond}$ under this projection map $\P^\calA\to\P^3$ 
is a rational surface in $\P^3$.
An affine part of the real points of this surface is shown in Figure~\ref{F:pillow}.
\begin{figure}[htb]
\[
  \includegraphics[width=250pt]{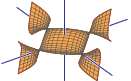}
\]
\caption{The double pillow.}
\label{F:pillow}
\end{figure}

In the coordinates $[w:x:y:z]$ for $\P^3$ this surface has the implicit equation
\[
  (x^2-y^2)^2 - 2x^2w^2 - 2y^2w^2 -16z^2w^2 + w^4\ =\ 0\,.
\]
and its dense torus has parametrization
\[
   [w:x:y:z]\ =\ 
     [{\textstyle s+t+\frac{1}{s}+\frac{1}{t}\,:\  s-\frac{1}{s}\,:\ 
      t-\frac{1}{t}\,:\ 1}]\,.
\]
It has curves of self-intersection along the lines $x=\pm y$ in the plane
at infinity ($w=0$).
As the self-intersection is at infinity, this affine surface is a good
illustration of the real points of the toric variety $X_{\MySmDiamond}$, and 
so we refer to this picture to describe $X_{\MySmDiamond}$. 

This surface contains four lines $x\pm y=\pm1$ and 
their complement is the dense torus in $X_{\MySmDiamond}$.
The complement of any three lines is the piece $V_\tau$ corresponding to a ray $\tau$.
Each of the four singular points is a singular point of one cone
$V_\sigma$, which is obtained by removing the two lines not meeting
that singular point.
Finally, the action of the group $\{(\pm1,\pm1)\}\subset(\C^*)^2$ on the real points 
$X_{\MySmDiamond}(\R)$ may also be seen from this picture.
Each singular point is fixed by this group.
The element $(-1,-1)$ sends $z\mapsto -z$, interchanging the top and bottom
halves of each piece, while the elements $(1,-1)$ and $(-1,1)$ interchange the
central `pillow' with the rest of $X_{\MySmDiamond}(\R)$.
In this way, we see that $X_{\MySmDiamond}(\R)$ is a `double pillow'.\smallskip

The nonnegative part $X_{\MySmDiamond}(\R_\geq)$ of $X_{\MySmDiamond}(\R)$ is also seen in Figure~\ref{F:pillow}. 
The upper part of the middle pillow is the part of $X_{\MySmDiamond}(\R)$
parameterized by $\R_>^2$, and so its closure is just a square, but with
singular corners obtained by cutting a cone into two pieces along a plane of
symmetry. 
This is $X_{\MySmDiamond}(\R_\geq)$.
In fact, the orthogonal projection to the $x,y$-plane identifies $X_{\MySmDiamond}(\R_\geq)$ with the polygon
$\raisebox{-1.5pt}{\MyDiamond}$. 
This is also a consequence of Lemma~\ref{L:Birch}.
The composition of the projection $\P^\calA\to\P^3\to\P^2$, where the last is the orthogonal projection to the $x,y$-plane,
is the projection map $\mu_{\calA}$, at least on $X_{\MySmDiamond}(\R_\geq)$.
From the symmetry of this surface, we see that $X_{\MySmDiamond}(\R)$ is
obtained by gluing four copies of the polygon $\raisebox{-1.5pt}{\MyDiamond}$ together along their
edges to form two pillows attached at their corners.
(The four `antennae' are actually the truncated corners of the second
pillow---projective geometry can play tricks on our affine intuition.)

\subsection*{Exercises}

\begin{enumerate}[1.] 
 \item 
     Show that the lattice octahedron~\eqref{Eq:SomePolytopes} is the intersection 
      of the eight half spaces,  $\{(x,y,z)\in\R^3\mid \pm x\pm y\pm z\leq 1\}$, 
      one for each choice of the three signs $\pm$.

 \item 
     Let $\calA \subset \R^n$ be a finite set, and let $P = \conv(\calA)$ be its convex hull, a polytope. 
     For $w\in \R^n$, recall that 
 \[
        \calA_w\ =\ \{a \in \calA \ |\ w\cdot a = h_p(w)\}\,.
 \]
      Let $P_w$ be the face of $P$ exposed by $w$.
      Show that 
 \[
     \calA_w\ =\  \calA \cap P_w,.
 \]

 \item 
    Using the notation from the previous exercise, prove that $P_w = \conv(\calA_w)$.

 \item    Prove that a vertex (face of dimension zero) in a polyhedron $P$ is \demph{extreme}
   in that it does not lie in the convex hull of other points of $P$.
   Deduce that a polytope is the convex hull of its vertices.
 
 \item 
     Show that every face of a cone is a cone, and that the minimal face is its lineality space.

 \item 
     Let $P\subset\R^n$ be a polytope.
     Show that for $v,w\in \R^n$, the relation 
 \[
     v\ \sim\ w \ \Longleftrightarrow\ P_v=P_w
 \]
   is an equivalence relation.
   Show that the closure of an equivalence class is a cone, and if $P$ is in integer polytope, the cone is rational.
   (Hint: express an equivalence class in terms of the vertices of $P$.)

 \item Determine the cones of the normal fan to the lattice cube, which you may take to be  the convex hull of the vectors 
       $\MySmTVec{\pm1}{\pm1}{\pm1}$, for all eight choices of $\pm$.

 \item Prove that the linear span of a cone $\sigma$ in a rational fan $\Sigma\subset\N_\R$ is spanned by its
   intersection with $N$.

 \item The final paragraph in the proof of Lemma~\ref{L:FirstCheck} is a bit dense.   Fill in the details.

 \item For each rational fan in $\R^2$ below, carry out the construction of the toric variety associated to the fan.
\[
   \includegraphics{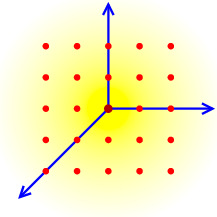}\qquad\quad
   \includegraphics{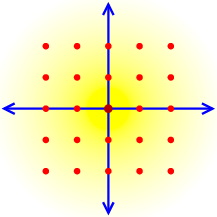}
\]
 Do you recognize either of these varieties?


\end{enumerate}

\section{Bernstein's Theorem and Mixed Volumes}\label{S:Bernstein}

Bernstein~\cite{Bernstein} gave a formula for the number of solutions to a system of polynomials where different
polynomials may have different supports, generalizing Kushnirenko's Theorem.
Bernstein's formula is in terms of Minkowski's mixed volume.
We first review mixed volume, and then give a proof of Bernstein's theorem, adapted from his paper, but using some
elementary notions from tropical geometry.
The discussion of mixed volumes is based on the pages 116--118 in~\cite{Ewald}.

\subsection{Mixed Volumes} \label{SS:MV}
Recall that for a polytope $P\subset\R^n$, $\Vol(P)$ is its volume with respect to the standard Euclidean metric on $\R^n$. 
Write $\Vol_n(P)$ if we need to emphasize the ambient space of $P$.
In particular, if $\dim P<n$, then $\Vol_n(P)=0$.
If $\dim P=m$, then $\Vol_m(P)$ is taken to be its volume in its $m$-dimensional affine span.
(This is used in the proof of Theorem~\ref{Th:Minkowski}.)

We consider two constructions involving polytopes.
Let $P,Q\subset\R^n$ be polytopes and $\lambda\geq 0$ a real number.
Then we may scale $P$ to obtain another polytope, 
\[
   \defcolor{\lambda P}\ :=\ \{ \lambda x \mid x\in P\}\,.
\]
The \demph{Minkowski sum} of $P$ and $Q$ is
\[
   \defcolor{P+Q}\ :=\ \{ x+y \mid x\in P, y\in Q\}\,.
\]
Note that $P+P=2P$.

\begin{example}
 Suppose that $\calA$ is represented by the matrix $(\begin{smallmatrix}0&1&1&2&2\\1&0&2&0&1\end{smallmatrix})$ and 
 $\calB$ is represented by the matrix $(\begin{smallmatrix}0&1&1&2\\0&1&2&1\end{smallmatrix})$, and set 
 $P:=\conv(\calA)$ and $Q:=\conv(\calB)$. 
 Then 
 $P+Q=\conv(\calA+\calB)=\conv(\calC)$, where 
 $\calC$ is represented by the matrix $(\begin{smallmatrix}0&1&1&2&2&4&4\\1&0&3&0&4&1&2\end{smallmatrix})$.
 We display these polytopes and their Minkowski sum in Figure~\ref{F:MinkowskiSum}.\qed
\end{example}
 \begin{figure}[htb]
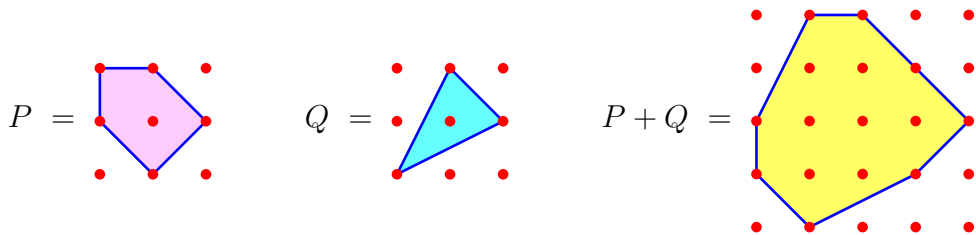

 \[
   P\ =\ \raisebox{-20pt}{\includegraphics{figures/LatticeP.eps}}
   \qquad\quad
   Q\ =\ \raisebox{-20pt}{\includegraphics{figures/LatticeQ.eps}}
   \qquad\quad
   P+Q\ =\ \raisebox{-40pt}{\includegraphics{figures/LatticeP+Q.eps}}
 \]
 \caption{Minkowski sum of two polygons.}
 \label{F:MinkowskiSum}
\end{figure}

Given polytopes $P_1,\dotsc,P_r\subset\R^n$  and nonnegative real numbers $\lambda_1,\dotsc,\lambda_r$, define
 \begin{equation}
 \label{Eq:scaledMinkSum}
  \defcolor{P(\lambda)} \ :=\ \lambda_1 P_1 + \dotsb + \lambda_r P_r\,.
 \end{equation}
The following lemma is left for you to prove in Exercise~\ref{Exer:Minkowski_face}.

\begin{lemma}
 \label{L:Minkowski_face}
 For any vector $w\in\R^n$, the support function $h_{P(\lambda)}(w)$ is the linear function
 $\lambda_1 h_{P_1}(w)+\dotsb+\lambda_r h_{P_r}(w)$, and 
\[
   P(\lambda)_w\ =\ \lambda_1 P_{1,w} + \dotsb + \lambda_r P_{r,w}\,.
\]
 If $P(\lambda)_w$ is a facet of $P(\lambda)$ for one choice of $\lambda_1,\dotsc,\lambda_r$ with all $\lambda_i>0$, then 
 $P(\lambda)_w$ is a facet of $P(\lambda)$ for any $\lambda_1,\dotsc,\lambda_r$ with all $\lambda_i>0$.
\end{lemma}

We prove the main result about the volume of the scaled Minkowski sum~\eqref{Eq:scaledMinkSum}.

\begin{theorem}[Minkowski]
 \label{Th:Minkowski}
 Let $P_1,\dotsc,P_r\subset\R^n$ be polytopes.
 For  nonnegative $\lambda_1,\dotsc,\lambda_r$, $\Vol_n(P(\lambda))$ is a homogeneous polynomial of degree $n$ in
 $\lambda_1,\dotsc,\lambda_r$.
\end{theorem}

\begin{proof}
 Suppose first that $n=1$.
 Then each $P_i$ is an interval $[a_i,b_i]$ with $a_i\leq b_i$ so that  
 $P(\lambda)=[\lambda_1 a_1+\dotsb+\lambda_r a_r,\lambda_1 b_1+\dotsb+\lambda_r b_r ]$, and we have 
\[
   \Vol_1(P(\lambda))\ =\ 
    \sum_{i=1}^r \lambda_ib_i - \sum_{i=1}^r \lambda_ia_i\ =\ 
    \sum_{i=1}^r \lambda_i(b_i-a_i)\ =\    \sum_{i=1}^r \lambda_i \Vol_1(P_i)\ ,
\]
 which is homogeneous of degree 1 in $\lambda_1,\dotsc,\lambda_r$.

 Now suppose that $n>1$.
 As volume is invariant under translation, we will make some assumptions for the purpose of computation.
 For a given $w\in\R^n$ and all $i$, we may assume that $0$ lies in the face $P_{i,w}$ of $P_i$ exposed by $w$.
 Then each $P_{i,w}$ as well as $P(\lambda)_w$ lies in the hyperplane annihilated by $w$, which is isomorphic to $\R^{n-1}$.
 By induction on dimension, we may assume that 
 $\Vol_{n-1}(P(\lambda)_w)=\Vol_{n-1}(\lambda_1 P_{1,w} + \dotsb + \lambda_r P_{r,w})$ is a homogeneous polynomial of degree
 $n{-}1$ in  $\lambda_1,\dotsc,\lambda_r$.
 This conclusion about $\Vol_{n-1}(P(\lambda)_w)$ remains true even if $0$ does not lie in any face $P_{i,w}$.
 
 Again translating $P(\lambda)$ if necessary, we may assume that $h_{P(\lambda)}(w)>0$.
 Then the pyramid
 $C_w$ with apex $0\in\R^n$ over the facet $P(\lambda)_w$ of $P(\lambda)$ has height $h_{P(\lambda)}(w)$ and
 therefore has  volume
\[
  \frac{1}{n}\cdot\frac{1}{\|w\|}h_{P(\lambda)}(w)\cdot\Vol_{n-1}(P(\lambda)_w)\,
\]
 which is a homogeneous polynomial of
 degree $n$ in  $\lambda_1,\dotsc,\lambda_r$, as $h_{P(\lambda)}(w)$ is linear in  $\lambda_1,\dotsc,\lambda_r$.
 Again using that volume is invariant under translation, now suppose that $0\in P(\lambda)$, and thus the support
 function of $P(\lambda)$ is nonnegative for all $w\in\R^n$.
 Then the pyramids over facets of $P(\lambda)$ form a  polyhedral subdivision of $P(\lambda)$, so that $\Vol(P(\lambda))$ is the sum
 of the volumes of these pyramids.  
 This completes the proof.
\end{proof}

Let us write the polynomial $\Vol(P(\lambda))$ as a tensor (nonsymmetric in $\lambda_1,\dotsc,\lambda_r$),
 \begin{equation}\label{Eq:MinkowskiTensor}
   \Vol(P(\lambda))\ =\ 
   \sum_{a_1,\dotsc,a_n=1}^{r} \MV(P_{a_1},P_{a_2},\dotsc,P_{a_n}) \lambda_{a_1}\lambda_{a_2}\dotsb\lambda_{a_n}\,,
 \end{equation}
 where the coefficients are chosen to be symmetric---for any permutation $\pi\in S_n$, we have 
\[
   \MV(P_{a_1},P_{a_2},\dotsc,P_{a_n})\ =\ \MV(P_{\pi(a_1)},P_{\pi(a_2)},\dotsc,P_{\pi(a_n)})\,.
\]
 The coefficient \defcolor{$\MV(P_{a_1}\dotsc,P_{a_n})$} is the \demph{mixed volume} of the polytopes
 $P_{a_1},\dotsc,P_{a_n}$.

\begin{lemma}
 \label{L:ThreeProperties}
 Mixed volumes satisfy the following properties.
 Let $P,Q,P_1,\dotsc,P_n\subset\R^n$ be polytopes.
 \begin{enumerate}
  \item Symmetry.  $\MV(P_{a_1},\dotsc,P_{a_n}) = \MV(P_{\pi(a_1)},\dotsc,P_{\pi(a_n)})$ for any permutation $\pi\in S_n$.
 
   \item Multilinearity.  For any nonnegative $\lambda,\mu$, we have 
\[
    \MV(\lambda P + \mu Q, P_2,\dotsc,P_n)\ =\ 
      \lambda \MV(P , P_2,\dotsc,P_n)\ + \  \mu \MV(Q, P_2,\dotsc,P_n)\,.
\]
  
   \item  Normalization.   $\MV(P,\dotsc,P)=\Vol_n(P)$.
 \end{enumerate}
\end{lemma}

The notion of (multi-)linearity in statement (2) is weaker than the usual notion.
Usually, a function $f(x)$ is linear in an argument $x$ if $f(\lambda x+\mu y)=\lambda f(x)+ \mu f(y)$ for arguments $x$
and $y$ and {\sl any} numbers $\lambda$ and $\mu$.
For mixed volume, the coefficients $\lambda$ and $\mu$ are nonnegative real numbers.

\begin{proof}
 Symmetry follows from the definition of mixed volume.
 For multilinearity, equate the coefficient of $\lambda_1\dotsb\lambda_n$ in the nonsymmetric
 expansions~\eqref{Eq:MinkowskiTensor} of 
\[
  \Vol( \lambda_1(\lambda P + \mu Q)+ P_2+\dotsb+P_n) 
  \ =\ 
 \Vol( \lambda_1\lambda P + \lambda_1\mu Q+ P_2+\dotsb+P_n)\,.
\]
 (For the first, $r=n$ and for the second, $r=n{+}1$ in~\eqref{Eq:MinkowskiTensor}.)
 Finally, for normalization, note that for $\lambda\geq 0$, 
 $\lambda^n\Vol(P)=\Vol(\lambda P)=\lambda^n \MV(P,\dotsc,P)$, with the first equality coming from the definition of volume
 and the second from the expansion~\eqref{Eq:MinkowskiTensor} defining mixed volume.
\end{proof}

These three properties characterize mixed volumes.

\begin{corollary}
 \label{C:MV_uniqueness}
 Mixed volume is the unique function of $n$-tuples of polytopes in $\R^n$ that satisfies the three properties of
 symmetry, multilinearity, and normalization of Lemma~\ref{L:ThreeProperties}.
\end{corollary}

\begin{proof}
 Let $L$ be a function of $n$-tuples of polytopes in $\R^n$ that satisfies the three properties of
 symmetry, multilinearity, and normalization of Lemma~\ref{L:ThreeProperties}.
 For any polytopes $P_1,\dotsc,P_n\subset\R^n$ and nonnegative $\lambda_1,\dotsc,\lambda_n$, we have
 $\Vol(P(\lambda))= L(P(\lambda),\dotsc,P(\lambda))$ by normalization.
 Expanding this using~\eqref{Eq:scaledMinkSum} and the multilinearity of $L$, we obtain
\[
   L(P(\lambda),\dotsc,P(\lambda))\ =\ 
   \sum_{a_1,\dotsc,a_n=1}^{n} L(P_{a_1},P_{a_2},\dotsc,P_{a_n}) \lambda_{a_1}\lambda_{a_2}\dotsb\lambda_{a_n}\,.
\]
 The equality of this sum with the sum~\eqref{Eq:MinkowskiTensor} and the symmetry of both $L$ and $\MV$ in their arguments
 completes the proof. 
\end{proof}

We give another formula for mixed volume and prove a stronger version of Corollary~\ref{C:MV_uniqueness}.
This involves a weaker condition than multilinearity, that of \demph{multiadditivity} in which the nonnegative coefficients
$\lambda$ and $\mu$ are both taken to be 1.
Given polytopes $P_1,\dotsc,P_n$ and $\emptyset\neq A\subset[n]$ write \defcolor{$P(A)$} for the Minkowski sum
$\sum_{i\in A} P_i$.

\begin{theorem}
 \label{Th:MV_formula}
 Let $\calP$ be a collection of polytopes in $\R^n$ that is closed under Minkowski sum.
 Suppose that $L$ is a function of $n$-tuples of polytopes in $\calP$ that is symmetric in its arguments and normalized (as
 in Lemma~\ref{L:ThreeProperties}), and that $L$ is multiadditive under Minkowski sum ($\lambda=\mu=1$ in
 Lemma~\ref{L:ThreeProperties}). 
 Then for any polytopes $P_1,\dotsc,P_n\in\calP$, we have
 \begin{equation}
   \label{Eq:MV_formula}
   n! L(P_1,\dotsc,P_n)\ =\ 
    \sum_{\emptyset\neq A\subset[n]} (-1)^{n-|A|}\,\Vol(P(A))\,.
 \end{equation}
 In particular, $L$ equals mixed volume, $L(P_1,\dotsc,P_n)=\MV(P_1,\dotsc,P_n)$.
\end{theorem}

\begin{example}
 If $P,Q,R$ are polytopes in $\R^3$, then $ 6 \MV(P,Q,R)$ equals 
\[
    \Vol(P+Q+R)-\Vol(P+Q)-\Vol(P+R)-\Vol(Q+R)+\Vol(P)+\Vol(Q)+\Vol(R)\,.
\]
For polygons $P,Q$, we have $2\MV(P,Q)=\Vol(P+Q)-\Vol(P)-\Vol(Q)$.
For the polygons in Figure~\ref{F:MinkowskiSum}, if we subdivide $P+Q$ as shown,
\[
   \includegraphics{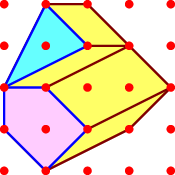}
\]
then $2\MV(P,Q)$ equals the combined areas of the four parallelograms, which is six.\qed
\end{example}

\begin{proof}[Proof of Theorem~\ref{Th:MV_formula}]
 Let $\emptyset\neq A\subset[n]$.
 Since $L$ is normalized, $L(P(A),\dotsc,P(A))$ equals $\Vol(P(A))$.
 Expand $L(P(A),\dotsc,P(A))$ using the multiadditivity of $L$ to obtain
 \begin{equation}
  \label{Eq:Partial_expansion}
   \Vol(P(A))\ =\ \sum_{a_1,\dotsc,a_n\in A} L(P_{a_1},\dotsc,P_{a_n})\,.
 \end{equation}
  Let $b_1,\dotsc,b_n$ be any sequence with $b_i\in[n]$ and set $B:=\{b_1,\dotsc,b_n\}$.
  Then $L(P_{b_1},\dotsc,P_{b_n})$ occurs in the sum~\eqref{Eq:Partial_expansion} if and only if $B\subset A$, and in that
  case, it appears with coefficient 1.

  Expand the right hand side of~\eqref{Eq:MV_formula} in terms of the function $L$ using~\eqref{Eq:Partial_expansion}. 
  Then for $b_1,\dotsc,b_n \in [n]$ the term $L(P_{b_1},\dotsc,P_{b_n})$ occurs with coefficient
\[
  \sum_{B\subset A\subset[n]} (-1)^{n-|A|}\ =\  (1-1)^{n-|B|}\ =\ 
   \begin{cases} 0 & \mbox{if }B\neq [n]\\
                 1 & \mbox{if }B=[n]
   \end{cases}\ .
\]
  Thus the right hand side of~\eqref{Eq:MV_formula} reduces to the sum of $L(P_{b_1},\dotsc,P_{b_n})$
  for $b_1,\dotsc,b_n$ distinct.  
  Each of these $n!$ terms are equal by symmetry, which completes the proof.
\end{proof}

\subsection{Bernstein's Theorem} \label{SS:Bernstein}
We begin with an example.

\begin{example}
The system $f=g=0$ of cubic sparse polynomials on $(\C^*)^2$, where 
%
%
%
 \begin{equation}\label{Eq:mixed_system}
   f\ :=\ \Blue{x + 2y + 3xy + 5x^2y + 7y^2 + 11xy^2}
  \quad\mbox{ and }\quad
   g :=\ \Magenta{1 + 3xy + 9x^2y + 27xy^2}\,,
 \end{equation}
has six solutions
\begin{eqnarray*}
  & (-0.21013, -0.44087)\,,\ 
    ( 0.94037, -0.13693)\,,\ 
    (-0.62796,  0.29688)\,,\ 
    (-1.1747,   0.36649)\,,&\\
  & 
    ( 0.85566 \mp 0.55260\sqrt{-1}, -0.36620\pm0.25941\sqrt{-1})\,,&
\end{eqnarray*}
and not $9=3\cdot 3$, which is the number predicted by B\'ezout's Theorem.
Figure~\ref{F:mixed_system} shows the curves defined by $f$ and $g$ in $\R^2$.
\begin{figure}[htb]
\[
   \begin{picture}(203,156)(-6,-3)
   \put(-3,-3){\includegraphics[height=156pt]{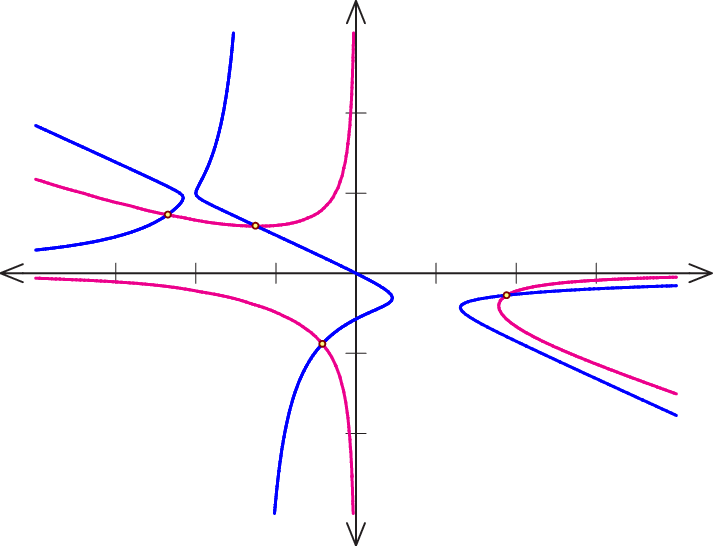}}
   \put(141.5,80.5){$1$}
   \put(105,117){$1$}   \put(105,26){$-1$}
   \put(10,121){\Blue{$f$}}   \put(54,137){\Blue{$f$}}
   \put(65,10){\Blue{$f$}}   \put(120,61){\Blue{$f$}}
   \put(-1,101){\Magenta{$g$}}   \put(10,65){\Magenta{$g$}}   \put(183,48){\Magenta{$g$}}
  \end{picture}
\]
\caption{Curves of the polynomial system~\eqref{Eq:mixed_system}.}
\label{F:mixed_system}
\end{figure}
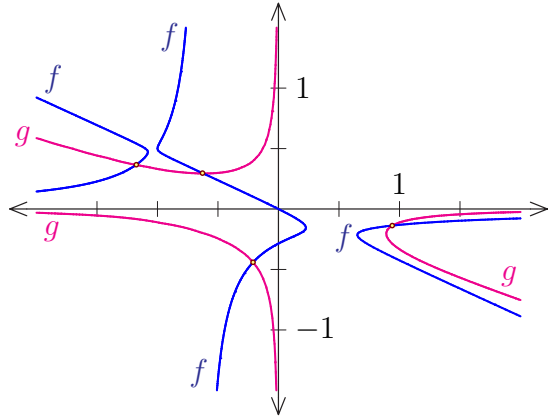
The Newton polytopes for $f$ and $g$ are the lattice polygons $P$ and $Q$ in Figure~\ref{F:MinkowskiSum}, respectively. 
Observe that the number of solutions is $2\MV(P,Q)$.
Exercise~\ref{Exer:mixed_system} asks you to compute the number of solutions for different pairs of polynomials
with the same support as $f$ and $g$~\eqref{Eq:mixed_system}.\qed
\end{example}

Bernstein's Theorem generalizes this observation.
As in Subsection~\ref{S:Kushnirenko}, for a finite set $\calA\subset M$, we identify the set of polynomials whose support
is a subset of $\calA$ with the vector space $\C^\calA$ of the possible coefficients of such polynomials.
We identify $\C^{\calA_1}\times\dotsb\times\C^{\calA_n}$ with the set of systems of polynomials with support
$(\calA_1,\dotsc,\calA_n)$, and $\C^{P_1}\times\dotsb\times\C^{P_n}$ the set of systems of polynomials with Newton polytopes 
$P_1,\dotsc,P_n$.

\begin{theorem}[Bernstein]
 The number of isolated solutions in $(\C^*)^n$, counted with multiplicity, of a system 
\[
   f_1(x)\ =\ f_2(x)\ =\ \dotsb\ =\ f_n(x)\ =\  0
\]
 of $n$ polynomials is at most $n!\MV(P_1,\dotsc,P_n)$, where $P_i$ is the Newton polytope of $f_i$.
 There is a dense open subset of\/ $\C^{P_1}\times\dotsb\times\C^{P_n}$ consisting of systems with  Newton polytopes 
 $P_1,\dotsc,P_n$ having exactly $n!\MV(P_1,\dotsc,P_n)$ solutions in $(\C^*)^n$, each isolated and occurring with multiplicity one. 
\end{theorem}

Given the results in Subsection~\ref{SS:MV}, particularly Theorem~\ref{Th:MV_formula}, our strategy for proving Bernstein's
Theorem will be to show that the number of solutions to a generic system with given supports depends only on the convex hull
of the supports, is symmetric, is multiadditive under Minkowski sum, and is normalized.
We first prove a lemma about this number for a generic system.

\begin{lemma}\label{L:genericSystems}
 Let $\calA_1,\dotsc,\calA_n$ be finite subsets of\/ $\Z^n$.
 Then there is a nonnegative integer $d$ and a nonempty open subset $U$ of\/ 
 $\C^{\calA_1}\times\dotsb\times\C^{\calA_n}$ consisting of polynomial systems such that 
 if $(f_1,\dotsc,f_n)\in U$ then $\calV(f_1,\dotsc,f_n)$ has exactly $d$ points and all are reduced.

 When $d=0$, if $\calV(f_1,\dotsc,f_n)\neq\emptyset$, then it has dimension at least one.
\end{lemma}

Write \defcolor{$d(\calA_1,\dotsc,\calA_n)$} for the number $d$ from the lemma.
Lemma~\ref{L:genericSystems} applies also to the unmixed systems of Kushnirenko's Theorem~\ref{Th:Kushnirenko}.

\begin{proof}
 Consider the incidence variety of solutions to systems of polynomials with supports $\calA_1,\dotsc,\calA_n$,
\[
   \defcolor{\Gamma}\ :=\ \{ (x,f_1,\dotsc,f_n)\in (\C^*)^n\times \C^{\calA_1}\times\dotsb\times\C^{\calA_n}
    \mid f_1(x)=\dotsb=f_n(x)\}\,.
\]
 For $x\in(\C^*)^n$, the set $\{f_i\in\C^{\calA_i}\mid f_i(x)=0\}$ is a hyperplane in $\C^{\calA_i}$, as $f_i(x)=0$ is a
 nonzero linear form on the coefficients of $f_i$.
 Thus the fiber of the map $\Gamma\to(\C^*)^n$ is the product of these $n$ hyperplanes and is thus a linear space of dimension 
 $\sum_{i=1}^n|\calA_i|\ -n$.
 This implies that $\Gamma$ is irreducible of dimension $\sum_{i=1}^n|\calA_i|$.
 
 The projection of $\Gamma$ to the other factor $\C^{\calA_1}\times\dotsb\times\C^{\calA_n}$ has 
 fiber over a point $(f_1,\dotsc,f_n)$ equal to the set of solutions $\calV(f_1,\dotsc,f_n)$.
 If this projection is surjective, then there is a positive integer $d$ and an open subset $U$ of the image consisting
 of points with exactly $d$ preimages---these are regular values of the projection. 
 (This is a consequence of Sard's Theorem and algebricity.)
 These are sparse systems with exactly $d$ solutions in $(\C^*)^n$, and each is reduced as the projection is regular over
 $U$. 

 If the map fails to be surjective, then the complement of its image contains an open subset $U$.
 Polynomial systems $(f_1,\dotsc,f_n)\in U$ have no solutions,  $\calV(f_1,\dotsc,f_n)=\emptyset$, and so $d=0$.
 This completes the proof of the first statement.
 Since the image of $\Gamma$ has dimension less than that of $\Gamma$, every fiber has positive dimension, proving the
 second statement. 
\end{proof}

Consider an unmixed system, where each polynomial $f_i$ has the same support, $\calA$.
Then Kushnirenko's Theorem~\ref{Th:Kushnirenko} implies that
$d(\calA,\dotsc,\calA)=n!\Vol_n(\conv(\calA))$.
Note also that the function $d$ is symmetric in its arguments.
To prove Bernstein's Theorem, we 
show that $d(\calA_1,\dotsc,\calA_n)$ depends only upon the convex hulls $\conv(\calA_1),\dotsc,\conv(\calA_n)$
and that it is multiadditive under Minkowski sum.

To understand multiadditivity, for a system  $(f_1,\dotsc,f_n)$ of polynomials, write 
\defcolor{$d(f_1,\dotsc,f_n)$} for the number of isolated points in $\calV(f_1,\dotsc,f_n)$ in the torus $(\C^*)^n$,
counted with multiplicity. 
It is clear that 
 \begin{equation}\label{Eq:product-sum}
   d(f\cdot g, f_2,\dotsc,f_n)\ \leq\ 
   d(f, f_2,\dotsc,f_n)\ +\ 
   d(g, f_2,\dotsc,f_n)\,,
 \end{equation} 
with equality when the system $(f\cdot g, f_2,\dotsc,f_n)$ has only isolated points.
More precisely, the inequality is strict when an isolated solution to one of the systems on the right hand side lies on a
positive-dimensional component defined by the other system.
Multiadditivity of $d(\calA_1,\dotsc,\calA_n)$ would follow from this observation~\eqref{Eq:product-sum}, if we
could show that
\[
   d(f, f_2,\dotsc,f_n)\ =\ d(\calA,\calA_2,\dotsc,\calA_n)
   \quad\mbox{and}\quad
   d(g, f_2,\dotsc,f_n)\ =\ d(\calB,\calA_2,\dotsc,\calA_n)\,,
\]
where $f$ has support $\calA$ and $g$ has support $\calB$, 
together imply that 
\begin{equation}\label{Eq:multiadditive}
    d(f\cdot g, f_2,\dotsc,f_n)\ =\ d(\calA+\calB,\calA_2,\dotsc,\calA_n)\,.
\end{equation}
This will follow by our next theorem, which characterizes the discriminant condition when
$d(f_1, f_2,\dotsc,f_n)<d(\calA_1,\calA_2,\dotsc,\calA_n)$, where $\calA_i=\supp(f_i)$ for $i=1,\dotsc,n$.

For a cocharacter $w\in N\simeq \Z^n$ and a Laurent polynomial $f$, write $f_w$ for the initial form $\ini_w(f)$ of $f$ in the
partial term order $\prec_w$.
Given a system $(f_1,\dotsc,f_n)$ of Laurent polynomials, consider the system of initial forms, $(f_{1,w},\dotsc,f_{n,w})$. 
Since $(t^w x)^a=t^{w\cdot a} x^a$, we have that $f_{i,w}(t^w x)=t^{h_{\calA_i}(w)} f_{i,w}(x)$ for each $i=1,\dotsc,n$.
Thus the variety $\calV(f_{1,w},\dotsc,f_{n,w})$ consists of orbits of $\C^*$ under its action on $(\C^*)^n$ given by the
cocharacter $t^w$ and is therefore either empty or of dimension at least one, by Lemma~\ref{L:genericSystems}.
In particular, translating each $f_{i,w}$ by an appropriate monomial, $(f_{1,w},\dotsc,f_{n,w})$ becomes a system of $n$
polynomials on the quotient $(\C^*)^n/\C^*_w\simeq(\C^*)^{n-1}$, where $\C^*_w\subset (\C^*)^n$ is the image of the
cocharacter $t^w$.  
We therefore expect that for general polynomials $f_1,\dotsc,f_n$ (given their support),
$\calV(f_{1,w},\dotsc,f_{n,w})=\emptyset$, by Lemma~\ref{L:genericSystems}. 

\begin{theorem}\label{Th:SparseDiscriminant}
 Let $(f_1,\dotsc,f_n)$ be a system of Laurent polynomials and set $\calA_i:=\supp(f_i)$.
\begin{enumerate}
 \item If $\calV(f_{1,w},\dotsc,f_{n,w})=\emptyset$ for all $w\in\Z^n\smallsetminus\{0\}$, then all points of $\calV(F)$ are
   isolated and we have $d(f_1,\dotsc,f_n)=d(\calA_1,\dotsc,\calA_n)$. 
 \item If for some $w\in\Z^n\smallsetminus\{0\}$,  $\calV(f_{1,w},\dotsc,f_{n,w})\neq\emptyset$, then 
    $d(f_1,\dotsc,f_n)<d(\calA_1,\dotsc,\calA_n)$ when we have $d(\calA_1,\dotsc,\calA_n)\neq 0$ and 
    $d(f_1,\dotsc,f_n)=0$ when $d(\calA_1,\dotsc,\calA_n)=0$. 
\end{enumerate}
\end{theorem}

A facial form $f_w$ of a polynomial corresponds to the subset $\calA_w$ of its support $\calA$.
As $\calA$ is finite, it has only finitely many subsets, so $f$ has only finitely many facial forms.
Consequently, there are only finitely many facial systems $(f_{1,w},\dotsc,f_{n,w})$ for a given system $(f_1,\dotsc,f_n)$.
Thus among {\it a priori} infinite set of conditions that $\calV(f_{1,w},\dotsc,f_{n,w})=\emptyset$ for all
$w\in\Z^n\smallsetminus \{0\}$, there are only finitely many irredundant conditions (one for each facial system).
Each of these is an algebraic condition on the coefficients of the system.
Thus general systems have $d(\calA_1,\dotsc,\calA_n)$ solutions, counted with multiplicity.

In fact, facial forms $f_w$ of a polynomial $f$ correspond to faces of the Newton polytope $\conv(\calA)$ of $f$ and thus
to cones in its normal fan.
More precisely, any two weights $w$ and $w'$ lying in the relative interior of the same cone $\sigma$ in the normal fan to
$\conv(\calA)$ give the same facial system, $f_w = f_{w'}$.
It follows that a facial system $(f_{1,w},\dotsc,f_{n,w})$ depends on which cone in the common
refinement of the normal fans of the polytopes $\conv(\calA_i)$ contains $w$ in its relative interior.

These observation imply the following.

\begin{corollary}\label{Co:convex_Hulls}
  We have that $d(\calA_1,\dotsc,\calA_n)$ is equal to $d(\overline{\calA_1},\dotsc,\overline{\calA_n})$, and so the number 
  $d(\calA_1,\dotsc,\calA_n)$ depends only upon the convex hulls of the supports.
\end{corollary}

Theorem~\ref{Th:SparseDiscriminant} also implies the multiadditivity of $d(\calA_1,\dotsc,\calA_n)$, and thus Bernstein's Theorem:
Let us call a system of polynomials $(f_1,\dotsc,f_n)$ \demph{Bernstein-general} if
$d(f_1,\dotsc,f_n)=d(\calA_1,\dotsc,\calA_n)$, where $\calA_i=\supp(f_i)$, for each $i=1,..,n$.
By our discussion, Bernstein-general systems are dense in $\C^{\calA_1}\times\dotsb\times\C^{\calA_n}$.
Projecting to the last $n-1$ factors shows that there exist an open subset $U$ of $\C^{\calA_2}\times\dotsb\times\C^{\calA_n}$
such that for $(f_2,\dotsc,f_n)\in U$, there exist $f_1\in\C^{\calA_1}$ such that $(f_1,\dotsc,f_n)$ is Bernstein-general.

Thus given supports $\calA,\calB, \calA_2,\dotsc,\calA_n$, there exist polynomials $f\in\C^\calA$, $g\in\C^\calB$, and
$f_i\in\C^{\calA_i}$ for $i=2,\dotsc,n$ such that both $(f,f_2,\dotsc,f_n)$ and $(g,f_2,\dotsc,f_n)$ are Bernstein-general.
By Theorem~\ref{Th:SparseDiscriminant}, no the facial system of either system has solutions.
As $(f\cdot g)_w = f_w \cdot g_w$, the inequality~\eqref{Eq:product-sum} implies that $(f\cdot g,f_2,\dotsc,f_n)$
is Bernstein-general, which then implies multiadditivity~\eqref{Eq:multiadditive}.
Thus the function $d(\calA_1,\dotsc,\calA_n)$, which only depends upon the convex hulls $P_i$ of the $\calA_i$, by
Corollary~\ref{Co:convex_Hulls},  satisfies the same properties as mixed volume of these convex hulls.
By Corollary~\ref{C:MV_uniqueness},  $d(\calA_1,\dotsc,\calA_n)=\MV(P_1,\dotsc,P_n)$, which is
Bernstein's Theorem.

\begin{proof}[Proof of Theorem~\ref{Th:SparseDiscriminant}.]
 Suppose first that $\dim(\calV(f_1,\dotsc,f_n))>0$, so that $\calV(f_1,\dotsc,f_n)$ has nonisolated solutions and thus
 $\dim\calV(f_1,\dotsc,f_n)\geq 1$. 
 It follows that the tropical variety $\Trop(\calV(f_1,\dotsc,f_n))$ of $\calV(f_1,\dotsc,f_n)$ has dimension at least one. 
 As $\Trop(\calV(f_1,\dotsc,f_n))$ is a rational cone, this implies that it contains a nonzero integer point
 $w\in\Trop(\calV(f_1,\dotsc,f_n))\cap\Z^n$ with $w\neq 0$. 
 But then the initial scheme $\ini_w(\calV(f_1,\dotsc,f_n))$ is nonempty, and therefore
 $\calV(f_{1,w},\dotsc,f_{n,w})\neq\emptyset$. 
 Thus if $\calV(f_{1,w},\dotsc,f_{n,w})=\emptyset$ for all $w\in\Z^n\smallsetminus\{0\}$, then all points of
 $\calV(f_1,\dotsc,f_n)$ are isolated. 

 Now suppose that all points of $\calV(f_1,\dotsc,f_n)$ are isolated.
 First suppose that $d(\calA_1,\dotsc,\calA_n)$ is nonzero and let $(g_1,\dotsc,g_n)$ be a system with support
 $\calA_1,\dotsc,\calA_n$ that  has $d(\calA_1,\dotsc,\calA_n)$ isolated solutions (and in fact exactly this number of
 solutions). 
 Consider the family of systems 
\[
   \defcolor{F_t}\ :=\ (f_1,\dotsc,f_n)\ +\ t(g_1,\dotsc,g_n)\,,
\]
 for $t\in\C^*$.
 For all $t$ with $|t|$ sufficiently large, this has $d(\calA_1,\dotsc,\calA_n)$ distinct solutions, and so 
 $\calV(F_t)$ defines a curve $C\subset \C^*\times(\C^*)^n$ whose fiber over a general point $t\in\C^*$ consists of 
 $d(\calA_1,\dotsc,\calA_n)$ points, and the difference $\calV(F_t)\smallsetminus C$ is contained in finitely many fibers
 over points of $\C^*$.  

 Let us consider the tropical variety $\Trop(C)\subset\R\times\R^n$ of $C$, which differs from $\Trop(\calV(F_t))$ only in
 some components with finite image in $\R=\Trop(\C^*)$.
\[
   \begin{picture}(278,116)(-56,-13)
    \put(0,0){\includegraphics{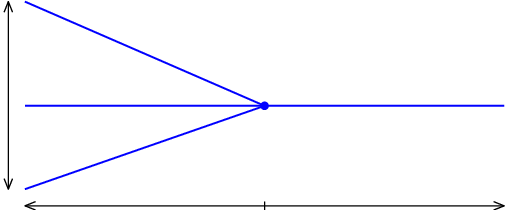}}
    \put(87,70){$\Trop(C)$}
    \put(165,55){$d(\calA_1,\dotsc,\calA_n)$}    \put(15,55){$d(f_1,\dotsc,f_n)$}
     \put(124.5,-11){$0$}
     \put(165,8){$\Trop(\C^*)=\R$}
     \put(-30,39){$=\R^n$}
     \put(-56,52){$\Trop(\C^*)^n$}
   \end{picture}
\]
  As $\ini_{(1,\bzero)}F_t=t(g_1,\dotsc,g_n)$, and $\ini_{(-1,\bzero)}F_t=(f_1,\dotsc,f_n)$, we see that $\Trop(C)$ has a
  ray of weight $d(\calA_1,\dotsc,\calA_n)$ in the direction $(1,\bzero)$ and a ray of weight $d(f_1,\dotsc,f_n)$ in the
  direction $(-1,\bzero)$. 
  Furthermore, the only ray with positive first coordinate is the ray with direction $(1,\bzero)$ as the fiber
  of $C$ over $t\gg 0$ consists of $d(\calA_1,\dotsc,\calA_n)$ points.
  By the balancing condition, $d(\calA_1,\dotsc,\calA_n)$ equals the sum of the weights of all rays with negative first
  coordinate.
  Thus $d(\calA_1,\dotsc,\calA_n)=d(f_1,\dotsc,f_n)$ if and only if there are no other rays $(-1,w)$ with negative first
  coordinate, which is equivalent to $\calV(f_{1,w},\dotsc,f_{n,w})=\emptyset$ for all nonzero $w\in\Z^n$.
  This completes the proof in the case that $d(\calA_1,\dotsc,\calA_n)\neq 0$

  To complete the proof, suppose that $d(\calA_1,\dotsc,\calA_n)=0$.
  By Lemma~\ref{L:genericSystems}, $\calV(f_1,\dotsc,f_n)$ is either empty or it has no isolated solutions, so that
  $d(f_1,\dotsc,f_n)=0$. 
\end{proof}

The invocation of tropical geometry in the proof may be avoided by appealing to asymptotic Puiseaux expansion of algebraic
curves as in Bernstein's original paper~\cite{Bernstein}.

\subsection*{Exercises}
\begin{enumerate}[1.]
 \item Show that for any sets $\calA,\calB\subset\R^n$, we have 
       $\conv(\calA) + \conv(\calB) = \conv(\calA+\calB)$.

\item  \label{Exer:Minkowski_face}
       Give a proof of Lemma~\ref{L:Minkowski_face}, including that the support function of $P(\lambda)$ is linear, as well 
        as that its faces are Minkowski sums of faces of its constituent polytopes.

 \item Let $f$ and $g$ be sparse polynomials.
       Prove that $\New(f\cdot g)=\New(f)+\New(g)$.
       Note that if $f$ and $g$ have support $\calA$ and $\calB$, respectively, then we only have
       $\supp(f\cdot g)\subset \calA+\calB$, as there may be cancellation. 
       (Suppose that $f=1+x$ and $g=1-x$.)
       However, there is no cancellation in the extreme points of $\calA$ and $\calB$, and this equality can be shown by
       considering the support functions.
 \item \label{Exer:mixed_system}
       Generate other pairs of polynomials with the same support as the polynomials in~\eqref{Eq:mixed_system}.
       For each pair, compute the degree of the ideal they generate.
       Can you prove this degree is six for generic coefficients?

\item
 Determine the Newton polytope of each polynomial, and the mixed volume of the Newton polytopes of each
 polynomial system.
 Check the conclusion of Bernstein's Theorem using a computer algebra system such as Macaulay2 or  
 Singular.
\begin{enumerate}
 \item \qquad${\displaystyle 1+2x+3y+4xy\ =\ 1-2xy+3x^2y-5xy^2 \ =\ 0}$.
         \vspace{4pt}

 \item \qquad${\displaystyle \begin{array}{rcl} 
                1 + 2x + 3y  - 5xy + 7x^2y^2&=&0\\ 
      1 - 2xy + 4x^2y + 8xy^2 - 16x^3y + 32xy^3 - 64x^2y^2&=&0
         \end{array}}$.
       \vspace{4pt}

 \item \qquad${\displaystyle \begin{array}{rcl} 
      2 + 5xy - x^2y - 6xy^2 + 4xy^3&=&0\\
      2x-y-2y^2-xy^2+2x^2y+x^2-5xy&=&0
     \end{array}}$.
   \vspace{4pt}

 \item \qquad${\displaystyle  \begin{array}{rcl}
        1+x+y+z+xy+xz+yz+xyz&=&0\\
        xy+2xyz+3xyz^2+5xz+7xy^2z+11yz+13x^2yz&=&0\\
        4 - x^2y + 2x^2z - xz^2  +2yz^2 - y^2z + 2y^2x -8xyz&=&0
     \end{array}}$.
   \vspace{20pt}
\end{enumerate}
\item Compute the mixed volume of the following pairs of lattice polygons.
\[
     (a)\quad \raisebox{-20pt}{\includegraphics{figures/LatticeDiamond.eps}}\hspace{10pt}
              \raisebox{-20pt}{\includegraphics{figures/LatticeHexagon.eps}}
    \qquad
     (b)\quad \raisebox{-20pt}{\includegraphics{figures/LatticeV.eps}}\hspace{10pt}
              \raisebox{-20pt}{\includegraphics{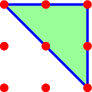}}
    \qquad
     (c)\quad \raisebox{-20pt}{\includegraphics{figures/LatticeV.eps}}\hspace{10pt}
              \raisebox{-30pt}{\includegraphics{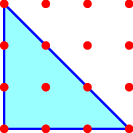}}
\]
\item Compute the mixed volume in $\R^3$ for the following three lattice polygons in the 
      $xy$-, $yz$-, and $xz$-planes, respectively.
    \[ 
      \includegraphics{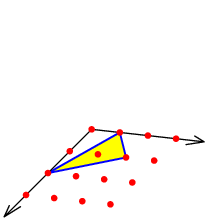}\qquad
      \includegraphics{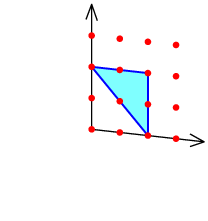}\qquad
      \includegraphics{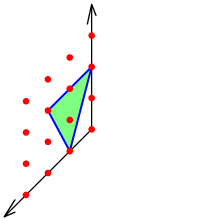}
    \]

\item  Compute the mixed volume in $\R^3$ of the following three lattice polytopes.
\[
  \includegraphics[height=80pt]{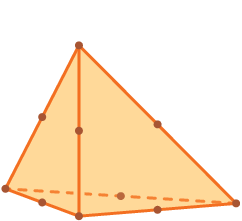} \qquad
  \includegraphics[height=60pt]{figures/Cube.eps}        \qquad
  \includegraphics[height=80pt]{figures/Octahedron.eps}
\]
\item
   Use Lemma~\ref{L:Minkowski_face} to prove that the facial system $(f_{1,w},\dotsc,f_{n,w})$ depends on the cone
   containing $w$ in the common refinement of the normal fans of the polytopes $\conv(\calA_i)$.

\item  Work out the details in the proof of Lemma~\ref{L:ThreeProperties} that were omitted in the proof sketch given.

\item
   Use  Bernstein's Theorem to deduce B\'ezout's Theorem: if $f_1,\dotsc,f_n$ are general polynomials of degree $n$ with
   $\deg(f_i)=d_i$ for $i=1,\dotsc,n$, then $d(f_1,\dotsc,f_n)=d_1d_2\dotsc d_n$.
   Hint: Determine $\defcolor{P_i}=\conv(\supp(f_i))$ for each $i$ and use the properties of mixed volume to compute
   $MV(P_1,\dotsc,P_n)$. 
   

\end{enumerate}

\providecommand{\bysame}{\leavevmode\hbox to3em{\hrulefill}\thinspace}
\providecommand{\MR}{\relax\ifhmode\unskip\space\fi MR }
\providecommand{\MRhref}[2]{%
  \href{http://www.ams.org/mathscinet-getitem?mr=#1}{#2}
}
\providecommand{\href}[2]{#2}

\end{document}